\documentclass[11pt]{amsart}

\usepackage{xypic}
\xyoption{all}
\usepackage{amstext}
\usepackage{amsthm}
\usepackage{amsopn}
\usepackage{amsfonts}
\usepackage{amsmath}
\usepackage{latexsym}
\usepackage{amscd}
\usepackage{amssymb}
\usepackage{amsmath}

\swapnumbers
\newtheorem{theorem}[subsection]{Theorem}
\newtheorem{definition}[subsection]{Definition}
\newtheorem{proposition}[subsection]{Proposition}
\newtheorem{lemma}[subsection]{Lemma}
\newtheorem{corollary}[subsection]{Corollary}
\newtheorem{claim}[subsection]{Claim}
\newtheorem{sublemma}[subsection]{Sublemma}

\theoremstyle{definition}

\newtheorem{remark}[subsection]{Remark}
\newtheorem{example}[subsection]{Example}

\newtheorem{examples}[subsection]{Examples}

\numberwithin{equation}{subsection}

\def\limind{\mathop{\oalign{lim\cr
\hidewidth$\longrightarrow$\hidewidth\cr}}}
\def\limproj{\mathop{\oalign{lim\cr
\hidewidth$\longleftarrow$\hidewidth\cr}}}

\def\gg{\mathfrak g}
\def\tt{\mathfrak t}
\def\hh{\mathfrak h}

\def\GG{\mathfrak G}
\def\EE{\mathfrak E}
\def\FF{\mathfrak F}
\def\XX{\mathfrak X}
\def\YY{\mathfrak Y}
\def\HH{\mathfrak H}
\def\PP{\mathfrak P}
\def\TT{\mathfrak T}
\def\SS{\mathfrak S}
\def\CC{\mathfrak C}
\def\LL{\mathfrak L}
\def\UU{\mathfrak U}
\def\VV{\mathfrak V}
\def\WW{\mathfrak W}
\def\ZZ{\mathfrak Z}

\def\MM{\mathfrak M}
\def\NN{\mathfrak N}
\def\QQ{\mathfrak Q}

\def\D{\mathbf D}
\def\G{\mathbf G}

\def\T{\mathbf T}
\def\t{\mathbf t}

\def\mm{\mathfrak m}
\def\nn{\mathfrak n}
\def\ss{\mathfrak s}
\def\kalg{k\text{--}alg}
\def\Ralg{R\text{--}alg}
\def\et{\acute et}

\def\limproj{\mathop{\oalign{lim\cr
\hidewidth$\longleftarrow$\hidewidth\cr}}}
\def\limind{\mathop{\oalign{lim\cr
\hidewidth$\longrightarrow$\hidewidth\cr}}}

\newcommand{\bbZ}{{\mathbb Z}}

\newcommand{\Aut}{\operatorname{Aut}}

\newcommand{\Spec}{\operatorname{Spec}}

\newcommand{\Hom}{\operatorname{Hom}}

\newcommand{\rad}{\operatorname{rad}}

\newcommand{\simlgr}{\buildrel \sim \over \lra}

 \newcommand{\bAut}{\rm \bf Aut}

\newcommand{\wh}{\widehat}

\newcommand{\Gal}{\operatorname{Gal}}

\newcommand{\lra}{\longrightarrow}

\newcommand{\cal}{\mathcal}

\def\ol{\overline}
\def\Z{\mathbb Z}
\newcommand{\bmu}{{\pmb \mu}}
\newcommand{\bs}{{\pmb \sigma}}

\def\bB{\text{\rm \bf B}}

\def\bT{\text{\rm \bf T}}

\def\bQ{\text{\rm \bf Q}}

\def\bH{\text{\rm \bf H}}
\def\bG{\text{\rm \bf G}}
\def\bJ{\text{\rm \bf J}}
\def\dg{\mathbf g}
\def\dh{\mathbf h}
\def\ds{\mathbf s}
\def\dt{\mathbf t}
\def\da{\mathbf a}
\def\dH{\mathbf H}
\def\dG{\mathbf G}

\newcommand{\uG}{{\underline G}}
\newcommand{\uH}{{\underline H}}
\newcommand{\bOut}{{\rm \bf Out}}

\def\bN{\text{\rm \bf N}}
\def\bU{\text{\rm \bf U}}

\def\bL{\text{\rm \bf L}}
\def\bP{\text{\rm \bf P}}
\def\bS{\text{\rm \bf S}}
\def\bM{\text{\rm \bf M}}

\def\cL{\mathcal{L}}
\def\al{\alpha}
\def\lm{\lambda}
\def\dL{\Lambda}

\def\bR{\text{\rm \bf R}}

\begin{document}
\title[Conjugacy theorems]{Conjugacy theorems for loop reductive group 
schemes and Lie algebras}

\author{V. Chernousov}
\address{Department of Mathematics, University of Alberta,
    Edmonton, Alberta T6G 2G1, Canada}
\thanks{ V. Chernousov was partially supported by the Canada Research
Chairs Program
and an NSERC research grant} \email{chernous@math.ualberta.ca}

\author{P. Gille}
\address{UMR 5208 du CNRS -
Institut Camille Jordan - Universit\'e Claude Bernard Lyon 1,
43 boulevard du 11 novembre 1918,
69622 Villeurbanne cedex - France.}
\thanks{P. Gille a b\'en\'efici\'e du soutien du projet ANR Gatho,  
ANR-12-BS01-0005.}\email{gille@math.univ-lyon1.fr}


\author{A. Pianzola}
\address{Department of Mathematics, University of Alberta,
    Edmonton, Alberta T6G 2G1, Canada.
    \newline
 \indent Centro de Altos Estudios en Ciencia Exactas, Avenida de Mayo 866, (1084) Buenos Aires, Argentina.}
\thanks{A. Pianzola wishes to thank NSERC and CONICET for their
continuous support}\email{a.pianzola@math.ualberta.ca}

\begin{abstract}
 \noindent The conjugacy of split Cartan subalgebras in the finite-dimensional simple case 
(Chevalley) and in the symmetrizable Kac-Moody case (Peterson-Kac) are fundamental 
results of  the theory of Lie algebras. Among the Kac-Moody Lie algebras the affine
 algebras stand out. This paper deals with the problem of conjugacy for a class of 
algebras --extended affine Lie algebras-- that are in a precise sense higher 
nullity analogues of the affine algebras. Unlike the methods used by Peterson-Kac,
 our approach is entirely cohomological and geometric. It is deeply rooted on the 
theory of reductive group schemes developed by Demazure and Grothendieck, and on 
the work of Bruhat-Tits on buildings.
 
 The main ingredient of our conjugacy proof is the classification
 of loop torsors over Laurent polynomial rings, a result of its own interest. 
 
\end{abstract}

\maketitle


\section{Introduction} Let $\dg$ be a split simple finite-dimensional Lie 
algebra over a field $k$ of characteristic $0.$ From the work of Cartan and 
Killing one knows that $\dg$ is determined by its root system. The problem, 
of course, is that a priori the type of the root system may depend on the 
choice of split Cartan subalgebra.  One of the most elegant ways of 
establishing that this does not happen, hence that the type of the root 
system is an invariant of $\dg,$ is the conjugacy theorem of split Cartan 
subalgebras due to Chevalley: all split Cartan subalgebras of $\dg$ are 
conjugate under the adjoint action of $\dG(k)$  where $\dG$ is the split 
simply connected group corresponding to $\dg.$ 

Variations of this theme are to be found on the seminal work of Peterson and 
Kac on conjugacy of ``Cartan subalgebras" for symmetrizable Kac-Moody Lie 
algebras \cite{PK}. Except for the toroidal case, nothing is known about 
conjugacy for extended affine Lie algebras (EALAs for short); a fascinating 
class of algebras which can be thought of as higher nullity analogues of the 
affine algebras. 

The aim of this paper is two-fold. First, to show the existence and conjugacy of what we call Borel-Mostow subalgebras; an important class of 
 ``Cartan subalgebras'' of multiloop algebras (Borel-Mostow subalgebras are rather special. A general conjugacy result fails, as we show in \S 9). As an application of conjugacy we show that the root system attached
to a Lie torus is an invariant (see Theorem~\ref{invariant}). Second, it turns out
that to solve the conjugacy problem we are, out of 
necessity,  faced with the classification problem of loop reductive group schemes over 
a Laurent polynomials ring $R_n=k[t_1^{\pm 1},\ldots,t_n^{\pm 1}]$. Our second 
main result provides us with a  local-global
principle for classification of loop torsors over $R_n$, 
a result that we believe is of its own interest. The case $n=1$ 
was done in our paper \cite{CGP} and here we consider the general case.
For details and precise statements we refer 
to \S\S\ref{Acyclicity1} and \ref{Acyclicity2}.

\smallskip

The philosophy that we follow is motivated by two assumptions:
\smallskip

(1) The affine Kac-Moody and extended affine Lie algebras are among the 
most relevant infinite-dimensional Lie algebras today.

(2) Since the affine and extended affine algebras are closely related to 
finite-dimensional simple Lie algebras, a proof of conjugacy ought to exist that is 
faithful to the spirit of finite-dimensional Lie theory.
\smallskip

That this much is true for toroidal Lie algebras (which correspond to the 
``untwisted case" in this paper) has been shown in \cite{P1}. The present 
work is much more ambitious. Not only it tackles the twisted case, 
but it does so in arbitrary nullity. 

Some of the algebras covered by our result are related to extended affine 
Lie algebras, but our work depicts a more global point of view. 
For every $k$-algebra $R$ which is a normal ring   
it builds  
a bridge between ad-k-diagonalizable subalgebras of twisted forms of semisimple Lie 
algebras over 
$R$ (viewed as infinite-dimensional Lie 
algebras over the base field $k$), and split tori of the corresponding reductive 
group schemes over $R.$ 
Using this natural one-to-one correspondence,  
shown in Theorem \ref{maincorrespondence},
we are able to attach a cohomological obstacle to conjugacy which eventually leads to the proof
of our main conjugacy result in Theorem \ref{conj-loop}. 
The main ingredient of the proof of conjugacy 
is the classification of loop reductive torsors over Laurent polynomial rings 
given by Theorem \ref{secondmaintheorem}. 


\section{Generalities on multiloop algebras and forms}

\subsection{Notation and conventions}\label{conventions}

Throughout this work, with the exception of the Appendix, $k$ will denote a 
field of characteristic
$0$ and $\ol{k}$ an algebraic closure of $k.$ For integers $n\geq 0$ and 
$m>0$  we set
$$
R_n=k[t_1^{\pm 1},\ldots, t_n^{\pm 1}],\ \ K_n=k(t_1,\ldots,t_n), \,\,\ F_n=k((t_1))\cdots((t_n)),
$$
and 
$$
R_{n,m}=k[t_1^{\pm \frac{1}{m}},\ldots, t_n^{\pm \frac{1}{m}}], \ 
K_{n,m}=k(t_1^{\frac{1}{m}},\ldots, t_n^{\frac{1}{m}}), \  
F_{n,m}=k((t_1^{\frac{1}{m}}))\cdots ((t_n^{\frac{1}{m}})).
$$

The category of commutative associate unital algebras over $k$ will be denoted by $\kalg.$  If $\XX$ is a scheme over $\Spec(k),$ by an {\it $\XX$--group} we will understand a group scheme over $\XX.$ When $\XX = \Spec(R)$ for some object $R$ of $\kalg,$ we use the expression {\it $R$--group}.
If $R$ is an object in $\kalg$ we will denote the  corresponding multiplicative and additive groups 
by  $\G_{m,R}$ and  $\bG_{a,R}.$ 

We will use bold roman characters, e.g. $\dG$,  $\dg$ to denote $k$--groups and their Lie algebras. 
The notation $\GG$ and $\gg$ will be reserved for $R$--groups (which are usually not 
obtained from a $k$--group by base change) and their Lie algebras.

\subsection{Forms}\label{forms} 

Let $\dg$ be a finite-dimensional split semisimple Lie algebra over $k.$ Recall  that a Lie algebra $\mathcal{L}$ over $R$ is called a {\it form} of $\dg\otimes_k R$ (or simply a form of  $\dg$)
if there exists a faithfully flat and finitely presented $R$--algebra $\widetilde{R}$ such that
\begin{equation}\label{form}
\mathcal{L}\otimes_R \widetilde{R} \simeq (\dg\otimes_k R)
\otimes_R \widetilde{R} \simeq \dg\otimes_k \widetilde{R},
\end{equation}
where all the above are isomorphisms of Lie algebras over $\widetilde{R}.$ The set of isomorphism classes of such forms is measured
by the pointed set
$$
H^1_{fppf}\big(\Spec(R), {\rm{\bf Aut}}(\dg)_{\text  \rm R}\big)
$$
where ${\rm {\bf Aut}}(\dg)_{\text \rm R}$ is the $R$--group obtained by base
change from the $k$--linear algebraic group $\rm{{\bf Aut}}(\dg)$.
We have a split exact sequence of $k$--groups
\begin{equation}\label{autoseq}
1\longrightarrow \dG_{ad}\longrightarrow
\rm{{\bf Aut}}(\dg) \longrightarrow \rm{\bf Out}
(\dg) \longrightarrow 1
\end{equation}
where $\dG_{ad}$ is the adjoint group
corresponding to $\dg$ and $\rm{\bf Out}
(\dg)$ is the constant $k$--group corresponding
to the finite (abstract) group of symmetries of the Coxeter-Dynkin diagram of
$\dg$. 

By base change we obtain an analogous sequence over $R.$ 
In what follows we will 
denote 
$H^1_{fppf}\big(\Spec(R), {\rm{\bf Aut}}(\dg)_{\text \rm R}\big)
$ simply by $H^1_{fppf}\big(R, \rm{{\bf Aut}}(\dg)\big)
$ when no confusion is possible. Similarly for the Zariski and \'etale topologies, as 
well as for $k$--groups other than $\rm{{\bf Aut}}(\dg).$

\begin{remark} Since $\rm{{\bf Aut}}(\dg)$ is smooth
and affine over $\Spec(R)$
$$
H^1_{\et}\big(R,{\rm{\bf Aut}}(\dg)\big) \simeq
H^1_{\it fppf}\big(R,{\rm{\bf Aut}}(\dg)\big).
$$
\end{remark}

\begin{remark}\label{isoremark} Let  $R = R_{n}$ be as in \S\,\ref{conventions}. 
By the Isotriviality Theorem of \cite{GP2} the trivializing
algebra $\widetilde{R}$ in (\ref{form}) may be taken to be of the form
$$
\widetilde{R} := R_{n,m}\otimes_k \tilde{k}=
\tilde{k}[t_1^{\pm \frac{1}{m}},\ldots, t_n^{\pm \frac{1}{m}}]$$
for some $m$ and some Galois extension $\tilde{k}$ of $k$ containing all  
$m$-th roots of unity of $\ol{k}.$ The extension $\widetilde{R}/R$ is Galois. 
\end{remark}

 \subsection{Multiloop algebras}\label{multiloopalgebras} Assume now that $k$ is algebraically closed. We fix a compatible set of primitive 
$m$--th roots of unity $\xi_m ,$ namely such that  $\xi _{me} ^e = \xi_m
$ for all $e > 0.$ Let $R = R_n$ and $\widetilde{R} =  R_{n,m}$. Then $\widetilde{R}/R$ is Galois. Via our choice of roots of unity, we can identify $\Gal(\widetilde{R}/R)$ with ${(\Z/m\Z )}^n$ as follows: For each
 ${\bf e} = (e_1,\dots ,e_n)\in \Z^n$ the corresponding element $\ol{\bf e} = (\ol e_1,\cdots,\ol e_n) \in \Gal(\widetilde{R}/R)$  acts on $\widetilde{R}$
via $
^{\ol {\bf e}} (t^{\frac{1}{m}}_i) = \xi  ^{e_i}_{m}
t^{\frac{1}{m}}_i.$
\smallskip

The primary example of  forms  $\cL$ of $\dg \otimes_k R$ which are trivialized by a Galois extension $\widetilde{R}/R$ as above are the multiloop algebras based on $\dg.$ These are defined as follows. Consider an
$n$--tuple $\bs = (\sigma  _1,\dots,\sigma  _n)$  of commuting
elements of ${\rm Aut}_k(\dg)$ satisfying $\sigma  ^{m}_i = 1.$ For
each $n$--tuple $(i_1,\dots ,i_n)\in \Z^n$ we consider the
simultaneous eigenspace
$$
\dg_{i_1 \dots i_n} =\{x\in \dg:\sigma  _j(x) = \xi  ^{i_j}_{m} x \,
\, \text{\rm for all} \,\, 1\le j\le n\}.
$$
Then $\dg = \sum \dg_{i_1 \dots i_n}, $ and $\dg = \bigoplus \dg_{i_1 \dots
i_n}$ if we restrict the sum to those $n$--tuples $(i_1,\dots ,i_n)$
for which  $0 \leq i_j < m_j,$ where $m_j$ is the order of $\sigma_j.$

The  {\it multiloop algebra based on $\dg$ corresponding to $\bs$}, commonly denoted by
$L(\dg,\bs),$ is defined by
$$
L(\dg,\bs) = \bigoplus_{{(i_1,\dots ,i_n)\in \Z^n}}
\dg_{i_1\dots i_n} \otimes t_1^{\frac{i_1}{m}} \dots
t^{\frac{i_n}{m}}_n \subset \dg\otimes _k \widetilde{R} \subset \dg \otimes_k \ol{R}_{\infty}
$$

\noindent where $\ol{R}_{\infty} = 
\limind \ol{k}[t_1^{\pm \frac{1}{m}},\ldots, t_n^{\pm 
\frac{1}{m}}]$.\footnote{The ring $\ol{R}_{\infty}$ is a useful artifice that allows us to 
see  {\it all} multiloop algebras based on a given $\dg$ as subalgebras of one Lie algebra.} 
Note that $L(\dg,\bs),$ which does not depend on the choice of common period $m,$ is not 
only a $k$--algebra (in general infinite-dimensional), but also naturally an $R$--algebra. 
A rather simple calculation shows that 
$$ L(\dg,\bs) \otimes _R \widetilde{R} \simeq \dg \otimes _k \widetilde{R} \simeq (\dg\otimes _k
R)\otimes _R \widetilde{R}.$$
\noindent Thus $L(\dg,\bs)$ corresponds to a torsor over $\Spec(R)$ under ${\bAut}(\dg)$.  

It is worth to point out that the cohomological information is always about the twisted forms viewed as algebras over $R$ (and {\it not} $k$). In practice, as the affine Kac-Moody case illustrates, one is interested in understanding  these algebras as objects over $k$  (and {\it not} $R$). We  find  in Theorem \ref{maincorrespondence} a bridge between these two very different and contrasting kinds of mathematical worlds.

\section{Preliminaries I: Reductive  group schemes}\label{prelimI}

\subsection{Some terminology}\label{terminology}


Let $\XX$ be a $k$-scheme. 
A {\it reductive} $\XX$--group is to be understood in the sense of
\cite{SGA3}. In particular, a reductive $k$--group is a reductive
{\it connected}$\,$ algebraic group defined over $k$ in the sense of
Borel. 
We recall now
two fundamental notions about reductive $\XX$--groups.

\begin{definition} Let $\GG$ be a reductive $\XX$--group. We say that $\GG$ is 
{\rm reducible}
if $\GG$ admits a proper parabolic subgroup $\PP$ which has a
Levi subgroup, and {\rm irreducible}
otherwise.
\end{definition}
\begin{definition}
We say that  $\GG$ is {\rm isotropic}
if 
$\GG$ admits a subgroup isomorphic to $\bG_{m,\XX}$. 
Otherwise we say that $\GG$ is
{\rm anisotropic.}
\end{definition}

We denote
 by ${\rm \bf Par}(\GG)$ the $\XX$--scheme of
parabolic subgroup of $\GG$. This scheme
is smooth and projective over $\XX$ \cite[XXVI, 3.5]{SGA3}. Since by definition $\GG$ is a
parabolic subgroup of $\GG,$ when $\XX$ is connected, to say that $\GG$  admits a proper parabolic subgroup is to say
that ${\rm \bf Par}(\GG)(\XX) \not =\{\GG \}.$

\begin{remark}\label{type}
{\rm  If $\XX$ is connected,
to each
parabolic subgroup $\PP$ of $\GG$ corresponds a ``type" ${\bf t} = {\bf
{t}}(\PP)$ which is a subset of the corresponding Coxeter-Dynkin diagram. Given a
type ${\bf t},$ the scheme ${\rm \bf Par}_{{\bf{t}}}(\GG)$ of parabolic
subgroups of $\GG$ of type ${\bf {t}}$ is also smooth and projective over
$\XX$ ({\it ibid.} cor.3.6).
}
\end{remark}

Let $\HH$ denote a reductive $\XX$--group. If $\TT$ is a subgroup of $\HH$ the 
expression ``$\TT$ is a maximal torus
of $\HH $"  has a precise meaning
(\cite[XII, D\'efinition $1.3$]{SGA3}). A maximal torus may or may not be split. 
If it is, we say that $\TT$ is a {\it split maximal torus.} This is in contrast 
with the 
concept of {\it maximal split torus} which we also need. This is  a closed 
subgroup of $\HH$ which is a split torus and which is not properly included in 
any other split torus of $\HH.$ Note that split maximal tori (even maximal tori) 
need not exist, while maximal split tori always do exist if $\XX$ is noetherian. 

If $\SS < \HH$ are $\XX$--groups and $\ss \subset \gg$ are their respective 
Lie algebras
we will denote by $Z_{\HH}(\SS)$ [resp. $Z_{\gg}(\ss)]$ the centralizer of 
$\SS$ in $\HH$ [resp. of $\ss$ in $\gg$]. If $\SS\subset \HH$ is 
a split torus then $Z_{\HH}(\SS)$ is a closed reductive subgroup
(see \cite[XIX, 2.2]{SGA3}). Also, if $\XX$ is connected
and $\TT$ a 
 torus of $\HH$ then $\TT$ contains a 
unique maximal
split subtorus $\TT_d$ (see \cite[XXVI, 6.5, 6.6]{SGA3}). 


We now recall and establish for future reference some basic useful facts.
\begin{lemma}\label{facts} Let $\HH$ be a reductive $\XX$--group and
$\SS\subset \HH$ a split torus.
Then 
there exists a parabolic subgroup $\PP\subset \HH$ such that $Z_{\HH}(\SS)$ is a
Levi subgroup of $\PP$. 
\end{lemma}

\begin{proof}
%
See \cite[XXVI, cor. 6.2]{SGA3}.
\end{proof}


\begin{lemma}\label{wellknown} Let $\SS$ be a split torus of $\HH,$ and 
let $\TT$ be the radical 
of the reductive group 
$\CC=Z_{\HH}(\SS).$\footnote{ Recall that the radical  of a reductive 
$\XX$--group is the  unique maximal torus of its centre \cite[XXII, 4.3.6]{SGA3}.}  
If $\XX$ is connected then $Z_{\HH}(\TT_d)=\CC$.
\end{lemma}\label{precentralizersplit}
\begin{proof}
Since $\TT$ is the centre of $\CC$ we have
$\CC\subset Z_{\HH}(\TT)$. Also, the inclusions
$\SS\subset \TT_d\subset \TT$ yield $$
Z_{\HH}(\TT)\subset Z_{\HH}(\TT_d)\subset Z_{\HH}(\SS)=\CC,
$$
whence the result.
\end{proof}

\begin{proposition}\label{flag} Let $\HH$ be a reductive group
scheme over  $\XX$. Assume $\XX$ is connected. Let $\SS$ be a split subtorus
of $\HH$ and let $\PP$ be a parabolic subgroup of $\HH$ containing 
$Z_\HH(\SS)$ as Levi subgroup. Then following  are equivalent:










\smallskip

{\rm 1)}  The reductive group scheme $Z_\HH(\SS)$ has no proper parabolic subgroups.


{\rm 2)}  $\PP$ is a minimal parabolic subgroup of $\HH$.



\smallskip

\noindent 
If $\SS$ is the maximal split subtorus of the radical of $Z_\HH(\SS)$
these two conditions are equivalent to 

\smallskip

{\rm 3)}
The reductive group scheme $Z_\HH(\SS)/\SS$ is anisotropic.
\end{proposition}


\begin{proof}








According to \cite[XXVI.1.20]{SGA3}, 
there is a bijective 
correspondence
$$
\Bigl\{\hbox{\,parabolics $\QQ$ of $\HH$ included in $\PP$\,} \Bigr\} \enskip  
<--> \enskip
\Bigl\{\hbox{\,parabolics $\MM$ of $Z_\HH(\SS)$}\,\Bigr\}
$$
Thus  the left handside  consists of one element if and only if so does the 
right handside. 
\end{proof}





\begin{proposition}\label{centralizer} Let $\GG$ be 
reductive group scheme over a connected base scheme $
\XX,$ $\SS$ a split subtorus of $\GG, $ and let $\gg$ and $\ss$ denote their respective Lie algebras. 
Then
 
 
 {\rm (1)} ${\rm Lie}\big(Z_{\GG}(\ss)\big) = Z_{\gg}(\ss).$
 
 {\rm (2)} 
$Z_\GG(\SS)$ is a Levi subgroup of $\GG$ and
 $Z_\GG(\SS) = Z_\GG(\ss)$.
 \end{proposition}
 
  \begin{proof}
  
  
  
  (1) This is a particular case of \cite[II th\'eo. 5.3.1(i)]{SGA3}.

(2) That $Z_\GG(\SS)$ is a  Levi subgroups of $\GG$ follows from Lemma \ref{facts}. To establish the equality $Z_\GG(\SS) = Z_\GG(\ss)$ we reason in steps.

(a) Assume ${\it \XX = \Spec(k) \,\, \,{\it  and }\,\,\, \GG \,\,\, simply \,\,connected}$: Then this is a result of Steinberg. See \cite[3.3 and 3.8]{[St75]}
 and  \cite[0.2]{[St75]}. 

(b) Assume ${\it \XX = \Spec(k) \,\, \,{\it and }\,\,\, \GG \,\,\,reductive}$: Embed
$\GG$ into ${\rm \bf SL}_n$ for a suitable $n.$ Then 
$$Z_\GG(\SS) =\GG \cap Z_{{\rm \bf SL_n}}(\SS) \,\, \,{\rm and }\,\,\,
Z_\GG(\ss) =\GG \cap Z_{{\rm \bf SL_n}}(\ss)$$ and we are reduced to the previous case.

(c) In general, 
we proceed by \'etale
descent. This reduces the problem to the case  $\SS \subset \TT
\subset \GG$ where $\GG$ is a Chevalley group and   $\TT$ its standard split
maximal torus. This sequence is obtained by base change
to $\XX$  from a similar sequence over $k$
by \cite[ VII cor. 1.6]{SGA3}. Over $k$ our equality holds.  Since both centralizers commute with base change the equality follows. 
\end{proof}

\section{Loop torsors and loop reductive group schemes}

Throughout this section $\XX$ will denote a connected and noetherian scheme 
over  $k$ and $\dG$  a $k$--group which is locally of finite 
presentation.\footnote{The case most relevant to our work is that of the 
group of automorphism of a reductive $k$--group.}

\subsection{The algebraic fundamental group}

If $\XX$ is a $k$--scheme and if
$a$ of is a geometric point of $\XX$ i.e. a morphism $a:\;\Spec(\Omega) \to \XX$
where $\Omega   $ is an algebraically closed field, we
denote the algebraic fundamental group of $\XX$ at $a$ by $\pi_1(\XX,a)$ (see \cite{SGA1} for details).

Suppose now that our $\XX$ is a geometrically connected $k$--scheme.
We will denote
$\XX \times_k {\ol k}$ by $\ol{\XX}.$  Fix a  geometric point 
$\ol a :\Spec(\ol k) \to \ol \XX.$ Let $a$ (resp. $b$) be the geometric point
of $\XX$ (resp. $\Spec(k)$)  given by the composite maps
 $a : \Spec(\ol k) \buildrel \ol a \over \to  \ol{\XX} \to \XX$ 
(resp. $b : \Spec(\ol k)
\buildrel \ol a \over \to   \XX \to \Spec(k)).$ Then by
\cite[th\'eo. IX.6.1]{SGA1} $\pi_1\big(\Spec(k), b\big) 
\simeq \Gal(k) := \Gal(\ol{k}/k)$
and the sequence
\begin{equation}\label{fundamentalexact}
1 \to \pi_1(\ol{\XX}, \ol a) \to  \pi_1(\XX,  a) \to
\Gal(k) \to 1
\end{equation}
is exact.

\subsection{The algebraic fundamental group of $R_n$}\label{aflaurent}

We refer the reader to \cite{GP2} and \cite{GP3} for details. 
The simply connected cover $\XX^{sc}$ of $\XX = \;\Spec\, (R_n)$ is $\Spec (\ol R_{n,\infty  })$where
$$ \ol R_{n,\infty  } = {\limind} \; \ol R_{n,m }$$ with $\ol R_{n,m  } = 
\ol k[t^{\pm
\frac{1}{m}}_1,\dots,t^{\pm \frac{1}{m}}_n].$ The ``evaluation at 1" 
provides a geometric point that we denote by $a.$ The algebraic 
fundamental group is best described as
\begin{equation}\label{FGLaurent}
\pi_1(\XX,a) =  \wh{\Z}(1)^n \rtimes
\;\Gal \,(k).
\end{equation}
\noindent where $\wh{\Z}(1)$ denotes the abstract group 
$ \limproj_m \,  \bmu_m(\ol k)$ equipped with the natural action of the 
absolute Galois group $\Gal(k).$

\subsection{Loop torsors}\label{looptorsors} 
Because of the universal nature of $\XX^{sc}$ we have a natural  group 
homomorphism
\begin{equation}
\label{Eq3} \bG(\ol k) {\longrightarrow} \bG(\XX^{sc}).
\end{equation}

The group $\pi_1(\XX,a)$ acts on $\ol{k},$ hence on $\bG(\ol{k}),$ via the
group homomorphism $\pi_1(\XX,a)\to\,\Gal \,(k)$ of
(\ref{fundamentalexact}). This action is continuous, and together
with (\ref{Eq3}) yields a map
$$
H^1\big(\pi_1(\XX,a),\bG(\ol{k})\big) \to
H^1\big(\pi_1(\XX,a),\bG(\XX^{sc})\big),
$$
where we remind the reader that these $H^1$ are defined in the 
``continuous" sense. 
On the other hand, by \cite[prop. 2.3]{GP3}  and basic
properties of torsors trivialized by Galois extensions we have a natural inclusion 
$$
\begin{aligned}
H^1\big(\pi_1(\XX,a),\bG(\XX^{sc})\big)
\subset H_{\et}^1(\XX,\bG).
\end{aligned}
$$

By means of the foregoing observations we make the following.
\begin{definition} \label{looptorsor} A torsor $\EE$
over $\XX$ under $\bG$ is called a {\rm loop torsor} if its
isomorphism class $[\EE]$ in $H_{\et}^1(\XX,\bG)$ belongs to the image of the
composite map
$$
H^1\big(\pi_1(\XX,a),\bG(\ol{k})\big) \to
H^1\big(\pi_1(\XX,a),\bG(\XX^{sc})\big)\subset H_{\et}^1(\XX,\bG).
$$
\end{definition}
\medskip

We will denote by $H^1_{loop}(\XX, \bG) $ the subset of
$H_{\et}^1(\XX, \bG)$ consisting of classes of loop torsors. 
They are given by (continuous) cocycles in the image of the natural 
map $Z^1\big(\pi_1(\XX,a),\bG(\ol{k})\big)  \to Z_{\et}^1(\XX,\bG),$ 
which  we call {\it loop cocycles.}

This fundamental concept is used in the definition of loop reductive groups 
which we will recall momentarily. 
The following examples  illustrate the immensely rich class of objects that 
fit within the language of loop torsors.

\begin{examples}\label{kloopexamples} (a) If $\XX =\,\Spec\,(k)$ then  
$H^1_{loop}(\XX, \bG)$ is nothing but the usual Galois
cohomology of $k$
with coefficients in $\bG.$

\vskip.2cm
(b) Assume that  $k$ is algebraically closed. Then the action of 
$\pi_1(\XX,a)$ on
$\bG(\ol{k})$ is trivial, so that
$$
H^1\big(\pi_1(\XX,a),\bG(\ol{k})\big)
=\;\Hom\big(\pi_1(\XX,a),\bG(\ol{k})\big)/\text{\rm Int}\, \bG(\ol{k})
$$
where the group $\text{\rm Int}\;\bG(\ol{k})$ of inner automorphisms of
$\bG(\ol{k})$ acts naturally on the right on 
$\Hom\big(\pi_1(\XX,a),\bG(\ol{k})\big).$ 
Two particular cases are important:

\vskip.2cm (b1) $\bG$ abelian: In this case
$H^1\big(\pi_1(\XX,a),\bG(\ol{k})\big)$ is just the group of
continuous homomorphisms from $\pi_1(\XX,a)$ to $\bG(\ol{k}).$

\vskip.2cm (b2) $\pi_1(\XX,a) =\widehat{\Z}(1)^n:$ In this case
$H^1\big(\pi_1(\XX,a),\bG(\ol{k})\big)$ is the set of conjugacy
classes of $n$--tuples $\bs = (\sigma_1,\dots,\sigma_n)$ of commuting
elements of finite order of $\bG(\ol{k}).$\footnote{ That the elements 
are of finite order follows from the continuity assumption.}

This last example is exactly the setup of multiloop algebras, and the 
motivation for the ``loop torsor"
terminology.
\end{examples}

\subsection{Geometric and arithmetic part of a loop cocycle}

By means of the decompositions (\ref{fundamentalexact}) and (\ref{FGLaurent}) we can think
of loop cocycles as being comprised of a geometric and an arithmetic
part, as we now explain. 

Let $\eta \in  Z^1\big(\pi_1(\XX,a), \bG(\ol{k})\big).$ The restriction
$\eta_{\mid \Gal(k)}$ is called the {\it arithmetic part} of $\eta$
and it is  denoted by $\eta^{ar}.$
  It is easily seen that $\eta^{ar}$ is in fact a cocycle in $Z^1\big(\Gal(k), \bG(\ol{k})\big)$. If $\eta$ is fixed in our discussion,
   we will at times denote  the cocycle $\eta^{ar}$ by the more
  traditional notation $z.$ In particular, for $s \in \Gal(k)$
 we write $z_s$ instead   of $\eta^{ar}_s.$ 

   Next we consider the restriction of $\eta$ to
  $\pi_1( \ol{\XX} ,\ol a)$ that we denote by $\eta^{geo}$
and called the {\it geometric part} of $\eta.$
We thus have a  map 
$$
\begin{CD}
\Theta \, \, : \, \,   & Z^1 \big(\pi_1(\XX,a), \bG(\ol{k})\big) @>>> Z^1\big(
\Gal(k), \bG(\ol{k})\big) \times
\Hom\big( \pi_1(\ol{\XX} ,\ol a), \bG(\ol{k})\big) \\ 
&\eta & \mapsto & \bigl( \quad \eta^{ar} \quad , \quad \eta^{geo} \quad \bigr)
\end{CD}
$$ 

The group $\Gal(k)$ acts on $\pi_1(\ol{\XX} ,\ol a)$ by conjugation. On
$\bG(\ol{k}),$ the Galois group $\Gal(k)$ acts on two different ways. 
There is the natural action arising from the action of $\Gal(k)$ on
$\ol{k}$, and there is also
the twisted action given by the cocycle $\eta^{ar} = z$.
Following  standard practice to view the abstract group
$\bG(\ol{k})$ as a $\Gal(k)$--module with the twisted action by $z$ we
write $_z{\bG(\ol{k})}.$

\begin{lemma}\label{Theta} The  map $\Theta$ described above yields  a 
bijection between
\newline $Z^1\big(\pi_1(\XX,a), \bG(\ol{k})\big)$ and couples 
$( z, \eta^{geo})$ with $z
\in  Z^1\big(
\Gal(k), \bG(\ol{k})\big)$ and 
$\eta^{geo} \in \Hom_{\Gal(k)}\big(
\pi_1(\overline{\XX}, \overline{a}), {_z\bG}(\ol{k})\big)$.
\end{lemma}

\begin{proof} See  \cite[lemma 3.7]{GP3}.
\end{proof}

\begin{remark}\label{easy} Assume that $\XX = \Spec(R_n).$ It is easy to 
verify that $\eta^{geo}$ arises from a unique $k$--group homomorphism
$$
_{\infty}\bmu =  \Bigl( \limproj \bmu_m\Bigr)^n \to  {_z\bG}
$$
\end{remark}

We finish this section by recalling some basic properties of the
twisting bijection (or torsion map) $\tau_z : H^1(\XX, {_z}\bG) \to H^1(\XX, 
\bG)$. 
Take a cocycle $\eta \in Z^1\big(\pi_1(\XX,a), \bG(\ol{k})\big)$ and
consider its corresponding pair $\Theta(\eta) = (z, \eta^{geo}).$ We
can apply the same construction to the twisted $k$--group ${_z\bG}.$
This would lead to a map $\Theta_z$ that will attach to a cocycle
$\eta' \in Z^1\big(\pi_1(\XX,a), {_z\bG}(\ol{k})\big)$ a pair 
$(z', {\eta'}^{\rm
geo})$ along the lines explained above.

\begin{lemma} \label{looptwist} Let $\eta \in Z^1\big(\pi_1(\XX,a), 
\bG(\ol{k})\big)$. With the above notation, the inverse of the twisting
map {\rm \cite{Se1}}
$$
\tau_z^{-1}:   Z^1\big(\pi_1(\XX,a), \bG(\ol{k})\big) \simlgr 
Z^1\big(\pi_1(\XX,a), {
_z\bG}(\ol{k})\big)
$$
satisfies $\Theta_z \circ \tau_z^{-1}(\eta) = (1, \eta^{\rm geo})$. \qed
\end{lemma}

\begin{remark}\label{klooptorsion} The notion of loop torsor behaves well  
under twisting by a Galois cocycle  
$z \in Z^1\big( \Gal(k), \bG(\ol{k})\big).$  Indeed   the torsion map $
\tau_z^{-1}: H_{\et}^1(\XX, \bG) \to
H_{et}^1(\XX, {_z\bG})
$
maps loop classes to loop classes. 
\end{remark}

\subsection{Loop reductive groups} \label{seckloopgroup}
Let $\HH$ be a reductive group scheme over $\XX.$  Since $\XX$ is connected, 
for all $x \in \XX$ the geometric fibers $\HH_{\ol x}$ are reductive 
group schemes of the same ``type"  \cite[XXII, 2.3]{SGA3}.  
By Demazure's  theorem there exists  a unique split reductive group 
$\bH$ over $k$ such that $\HH$ is a twisted form (in the \'etale topology of 
$\XX$) of $\HH_0 = \bH \times_k \XX.$ We will call $\bH$ the 
{\it Chevalley $k$--form of} $\HH.$ The $\XX$--group $\HH$ corresponds 
to a torsor $\EE$ over $\XX$ under the group scheme ${\bAut}(\HH_0),$ 
namely $\EE = {\rm \bf Isom}_{gr}(\HH_0, \HH).$  
We recall that $\bAut(\HH_0)$ is representable by a smooth and separated 
group scheme over $\XX$ by \cite[XXII, 2.3]{SGA3}. It is well-known that 
$\HH$ is then the  contracted product  $\EE \wedge^{\bAut(\HH_0)} \HH_0$ 
(see \cite{DG} III \S 4 n$^{\rm o}$3 for details). 

We now recall one of the central concepts needed for our work.

\begin{definition} \label{defkloopgroup}
We say that a group scheme $\HH$ over $\XX$ is  {\rm loop reductive}
if it is reductive  and if $\EE$  is
a loop torsor.
\end{definition}

\section{Preliminaries II: Reductive group schemes over a normal noetherian base}\label{prelimIII}

We begin with a useful variation of  Lemma \ref{wellknown} under some extra assumptions on our connected base $k$--scheme $\XX.$ 

\begin{lemma}\label{normalsplit} Assume that $\XX$ is normal noetherian and integral. Let $\HH$ be a reductive $\XX$--group. Then there exists an \'etale cover $(\UU_i)_{i=1,..,l} \to \XX$
such that :

\smallskip

(i) $\HH \times_\XX \UU_i$ is a split reductive $\UU_i$--group scheme, 


(ii) $\UU_i=\Spec(R_i)$ with $R_i$  a normal noetherian domain. 

(iii) If $\HH$ is a torus and $\XX = \Spec(R)$ there exists a Galois extension $\widetilde{R}/R$ that splits $\HH.$
 
\end{lemma}

\begin{proof} Since $\XX$ is normal noetherian, $\HH$ is a locally isotrivial 
group scheme \cite[XXIV.4.1.6]{SGA3}. We can thus cover  $\XX$
by affine Zariski open subsets $\XX_1, \dots, \XX_l$ where $\XX_i = \Spec(A_i)$ and such that there exists
 a finite \'etale cover $\VV_i \to \XX_i$ for $i=1,..,l$ which splits $\HH_{\XX_i}$. 
For each $i$, choose   a connected component $\UU_i$ of $\VV_i$.
 According to the classification of \'etale maps over $\XX$ 
(see \cite[18.10.12]{EGAIV}) 
we know
that $\UU_i$ is a finite \'etale cover of $\XX_i$ and that 
$\UU_i=\Spec(R_i)$  where $R_i$ is a normal domain.
Since $R_i$ is finite over the noetherian ring  $A_i$,
it is noetherian as well.

(iii) By \cite[X, th\'eo. 5.16]{SGA3} there exists a finite \'etale extension of $\XX$ that 
splits $\HH$. The result now follows by considering a connected component of this extension and basic properties of the algebraic fundamental group (see \cite[5.3.9]{Sz}).
\end{proof}

 \begin{remark}\label{localsplit} If $\XX$ is local, one single $\UU_i$ 
suffices. 
 \end{remark}

\begin{proposition}\label{centralizersplit} 
Let $\XX$ be 
normal and noetherian. Let $\HH$ be a reductive $\XX$--group,  $\PP\subset \HH$ 
be a parabolic
subgroup and $\LL\subset \PP$ a Levi subgroup.\footnote{The  existence of $\LL$ 
is automatic if the base scheme is affine by 
\cite[XXVI.2.3]{SGA3}} Let $\TT$ be
the radical of $\LL$  and $\TT_d$ 
its maximal split subtorus. Then $Z_\HH(\TT_d)=\LL.$
\end{proposition}

\begin{proof}
Since $\TT$ is the centre of $\LL$
we have $\LL\subset Z_\HH(\TT).$ The inclusion $\TT_d\subset \TT$ yields
$Z_\HH(\TT)\subset Z_\HH(\TT_d)$. Thus  we have $\LL\subset Z_\HH(\TT_d)$.
By the  Lemma below and by \cite[XXVI, prop. 6.8]{SGA3} the above inclusion is an equality locally in the Zariski topology, hence globally.
\end{proof}

\begin{lemma}\label{maxtor} Assume $\XX = \Spec(R)$ is affine and as in the 
Proposition.
Let  $x \in \XX$ and consider the localized ring $R_x.$ Then  $({\TT_{d}})_{R_x}$ is 
the maximal split subtorus of $\TT_{R_x}$. In particular, if $K$ denotes the quotient field 
of $R$ then $\TT_d \times_R K$ is the maximal split subtorus of $\TT \times_R K.$
\end{lemma}
\begin{proof} It suffices to show that $({\TT_d})_K$ is the maximal split
subtorus of $\TT_K$.  
Recall that $\TT$ is determined by its lattice of characters
$X(\TT)$ equipped with an action of ${\rm Gal}\,(\widetilde{R}/R),$
and that $\TT_d$ corresponds to the maximal sublattice
in $X(\TT)$ stable (elementwise) with respect to ${\rm Gal}\,(\widetilde{R}/R)$. Similar considerations apply to $\TT_K.$ 
It remains to note that $\TT_K$ and $\TT$ have the same lattices of characters and that ${\rm Gal}\,(\widetilde{R}/R) \simeq {\rm Gal}\,(\widetilde{K}/K)$ by  \cite[Ch5 \S2.2 theo.2]{Bbk}). \end{proof}

\begin{proposition}\label{G_m}
Let $\GG$ be a reductive group over
a normal ring\footnote{All of our normal rings are  hereon assumed to be integral and noetherian.}  $R$. If $\GG$  contains
a proper parabolic subgroup $\PP$ then it contains a split non-central subtorus
${\bf G}_{m,R}$.
\end{proposition}
\begin{proof}
We may assume that $\GG$ is semisimple. Since the base is affine, $\PP$ contains 
a Levi subgroup $\LL$. 
Let $\TT$ be the radical
of $\LL$ and $\TT_d$ its maximal split subtorus. By Proposition~\ref{centralizersplit}, $Z_\HH(\TT_d)=\LL.$ Hence $\TT_d\not=1$.
\end{proof}
\begin{corollary} For a reductive group scheme $\GG$ over a normal ring $R$
to contain 
a proper parabolic subgroup it is necessary and sufficient that it  contains
a non-central split subtorus.
\end{corollary}

\section{AD and  MAD subalgebras}\label{admad}

Let $R$ be an object in $\kalg$ and  $\GG$ be an $R$--group, i.e a  group scheme over $R.$ 
Recall (see \cite{DG} II \S4.1) that to $\GG$ we can attach an $R$--functor on Lie algebras $\mathfrak{Lie}(\GG)$ which associates to an object $S$ of $\Ralg$ the kernel of the natural map $\GG(S[\epsilon]) \to \GG(S)$ where $S[\epsilon]$ is the algebra of dual numbers over $S.$ Let ${\rm Lie}(\GG) = \mathfrak{Lie}(\GG)(R).$ This is an $R$--Lie algebra that will be denoted by $\gg$ in what follows. 

\begin{remark} If $\GG$ is smooth, the additive group of ${\rm Lie}(\GG)$ represents $\mathfrak{Lie}(\GG)$, that is  $\mathfrak{Lie}(\GG)(S) = {\rm Lie}(\GG) \otimes_R S$ as $S$--Lie algebras (this equality is strictly speaking a functorial family of canonical isomorphisms).
\end{remark}

If $S$ is in $\Ralg$, $g \in \GG(S)$ and $ x \in \mathfrak{Lie}(\GG)(S)$, then $gxg^{-1} \in \mathfrak{Lie}(\GG)(S)$ This last  product is computed in the group $\GG(S[\epsilon])$ where $g$ is viewed as an element of $\GG(S[\epsilon])$ by functoriality. The above defines an action of $\GG$ on $\mathfrak{Lie}(\GG)(S)$, called the adjoint action and denoted by $ g \mapsto {\rm Ad} (g).$ This action in fact induces an $R$--group homomorphism 
$${\rm Ad} : \GG \to \mathfrak{Aut}\big(\mathfrak{Lie}(\GG)\big)$$ whose kernel is the centre of $\GG.$

Given a $k$--subspace $V$ of $\gg$ consider the $R$--group functor $Z_{\GG}(V)$ defined by
\begin{equation}\label{centrals}
Z_{\GG}(V) :S \to \{g \in \GG(S): {\rm Ad}(g)(v_S) = v_S \ \mbox{for every \ } v\in V \}
\end{equation}
for all $S$ in $\Ralg,$ where $v_S$ denotes the image of $v$ in $\gg \otimes_R S.$

We will denote by $RV$ the $R$-span of $V$ inside $\mathfrak{g}$, i.e.
$RV$ is the $R$-submodule of $\mathfrak{g}$ generated by $V$. 
\begin{remark}\label{linearity} Note that $Z_{\GG}(V) = Z_{\GG}(RV).$ This follows from the fact that the adjoint action of $\GG$ on $\gg$ is ``linear" (in a functorial way). 
\end{remark}

We now introduce some of the central concepts of this work.

A subalgebra $\mathfrak{m}$ of the $k$-Lie algebra $\gg$
is called an AD {\it subalgebra} if the adjoint action of each element $x \in \mathfrak{m}$ on $\gg$ is $k$--diagonalizable, i.e. $\gg$ admits a $k$--basis
consisting of eigenvectors of $\text{\rm ad}_{\gg}(x).$
A maximal  AD subalgebra of $\gg,$ namely one which is not properly included in any other AD subalgebra of $\gg$ is called a MAD subalgebra of 
$\gg.$\footnote{It is not difficult to see that any such $\mathfrak{m}$
is necessarily abelian, so AD can be thought as shorthand for
abelian $k$--diagonalizable or ad $k$--diagonalizable.}

\begin{example} Let $\dG$ be a semisimple Chevalley $k$--group and $\bT$ its standard 
maximal split torus. Let $\dh$ be the Lie algebra of $\bT$; it is a split
Cartan subalgebra of $\dg$. For all $R$ we have  $\gg: = {\rm Lie}(\dG_R) = 
\dg \otimes_k R.$ 
Assume that $R$ is {\it connected}. Then $\mathfrak{m}=
\dh\otimes 1$ is a MAD subalgebra 
of $\gg$ by \cite[cor. to theo.1(i)]{P1}. We have $Z_{\dG_R}(\mm) = \bT_R.$ 

Note that $\mathfrak{m}$  is not its own normalizer.
Indeed $N_{\gg}(\mathfrak{m})= Z_{\gg}(\mathfrak{m})=\dh\otimes_k R$.
Thus $\dh \otimes 1$ is not a Cartan subalgebra of $\gg$ in the usual sense. 
However, in infinite-dimensional Lie theory -- for example, in the case of Kac-Moody 
Lie algebras -- 
these types of subalgebras do play the role that  the split Cartan subalgebras play in 
the classical theory. This is our motivation for studying conjugacy questions 
related to MAD subalgebras. 
\end{example}

\begin{remark}\label{m(S)} Let $\ss$ be an abelian Lie subalgebra of $\gg.$ Let ${\mm}_1$ and ${\mm}_2$
be two subalgebras of $\ss$ which are AD subalgebras of $\gg.$ Because $\ss$ is abelian their sum ${\mm}_1+{\mm}_2$ is  also an AD subalgebra of $\gg.$  By considering the sum of all such subalgebras we see that 
$\ss$ contains a {\it unique} maximal subalgebra $\mathfrak{m}(\ss)$
which is an AD subalgebra of $\gg.$ Of course this AD subalgebra 
need not be a MAD subalgebra of $\gg$.

We will encounter this situation when $\ss$ is the Lie algebra of a torus $\SS$ inside a reductive group scheme $\GG.$ In this case we denote  $\mathfrak{m}(\ss)$ by  $\mathfrak{m}(\SS).$
\end{remark}

\begin{remark}\label{adup} Let $\mm$ be an AD subalgebra of $\gg$. Then for any extension $S/R$ in $\kalg$  the image $\mm \otimes 1$ of $\mm$ in $\gg \otimes_R S$ is an AD subalgebra of $\gg \otimes_R S.$ Indeed if $x \in \mm$ and $v \in \gg$ are such that $[x, v] = \lambda v$ for some $\lambda \in k$, then $[x\otimes 1, v \otimes s] = v \otimes \lambda s = \lambda (v \otimes s)$ for all $s \in S.$ Thus $\gg \otimes_R S$ is spanned as a $k$--space by eigenvectors of $\text{\rm ad}_{\gg \otimes_R S}(x \otimes 1).$ Note that if the map $\gg \to \gg \otimes_R S$ is injective, for example if $S/R$ is faithfully flat, then we can identify $\mm$ with $\mm \otimes 1$ and view $\mm$ as an AD subalgebra of $\gg \otimes_R S.$

\end{remark}

The main thrust of this work is to investigate the question
of conjugacy of MAD subalgebras of $\gg$ when $\gg$ is a twisted form of $\dg \otimes_k R_n.$ The result we aim for is in the spirit of Chevalley's work, as explained in the Introduction. In the ``untwisted case''
the result is as expected.

\begin{theorem} Let $\dg$ be a split finite-dimensional semisimple Lie algebra over 
$k$ and $\dG$ the corresponding simply connected Chevalley group. Then all 
MAD subalgebras of $\dg\otimes_k R_n$ are conjugate to
$\dh\otimes 1$ under $\dG(R_n)$. 
\end{theorem}

This is a particular case of Theorem $1$ of \cite{P1} by taking
Cor 2.3 of \cite{GP2} into consideration. The proof is cohomological in nature,
which is also the approach that we will pursue here. As we shall see, the general twisted case holds many surprises in place.

We finish by stating and proving a simple result for future use.

\begin{lemma}\label{splitradical} Let $\G$ be a semisimple
algebraic
group over a field $L$ of characteristic $0$. Let $\T\subset \G$
be a torus and ${\T}_d$ be the (unique) maximal
split subtorus of $\T$. Set $\dg={\rm Lie}\,(\G)$, $\dt=
{\rm Lie}\,(\T)$ and ${\dt}_d={\rm Lie}\,({\T}_d)$. Then

\smallskip

\noindent
{\rm (i)} The adjoint action of ${\T}_d$ on $\dg$ is $L$--diagonalizable.
In particular, ${\dt}_d$ is  an {\rm  AD} subalgebra of $\dg$.

\smallskip

\noindent {\rm (ii)} ${\dt}_d$ is the largest subalgebra of $\dt$
satisfying the condition given in {\rm (i)}.
\end{lemma}
\begin{proof} Part (i) is clear. As for (ii) we may assume that $\dG$ is semisimple adjoint. Let ${\T}_a$
be the largest anisotropic subtorus of $\T$. The product morphism
${\T}_d\times {\T}_a \to \T$ is a central isogeny, hence
$\dt={\dt}_d\oplus {\dt}_a$ where ${\dt}_a={\rm Lie}\, ({\T}_a)$. We must
show that ${\dt}_a$ does not contain any nonzero element
whose adjoint action on $\dg$ is $L$--diagonalizable.
Let $h$ be such an element. Fix a basis $\{\,v_1,\ldots,v_n\,\}$
of $\dg$ and scalars $\lambda_i\in L$ such that
$$
[\,h, v_i\,]=\lambda_i v_i \ \ \forall\ 1\leq i \leq n.
$$
By means of this basis we identify ${\rm \bf GL}(\dg)$ with ${\rm \bf GL}_{n,L} $.
Consider the adjoint representation diagrams
$$
\T \hookrightarrow \G \stackrel{{\rm Ad}} \longrightarrow {\rm \bf GL}(\dg)\simeq
{{\rm \bf GL}}_{n,L}
$$
and
$$
\tt\hookrightarrow \dg \stackrel{\rm ad}{\longrightarrow} {\mathfrak{ gl}}\,(\dg)
\simeq {\mathfrak{gl}}_{n,L}.
$$
Since $\G$ is of adjoint type ${\rm Ad}$ is injective, so that we
can identify $\T$ with a subtorus, say $\widetilde{\T}$, of
${\rm \bf GL}_{n,L}$. Similarly for ${\T}_d$ and ${\T}_a$.
Since $\T\simeq \widetilde{\T}$ we see that $\widetilde{\T}_d$ and
$\widetilde{\T}_a$ are the maximal split and anisotropic parts of $\widetilde{\T}$.

Let ${\D}_n$ be the diagonal subgroup of ${\rm \bf GL}_{n,L}$.
By construction we see that
$$
{\rm ad}_{\gg}(h)\in {\rm Lie}\,({\D}_n) \cap {\rm Lie}\,(\widetilde{\T}_a)=
{\rm Lie}\,({\D}_n\cap \widetilde{\T}_a),
$$
this last by \cite[theo. 12.5]{[Hu]} since ${\rm char}(k)=0$. 
Thus ${\D}_n\cap \widetilde{\T}_a$
has dimension $>0$. But then the connected component of the identity
of ${\D}_n\cap \widetilde{\T}_a$ is a non-trivial split torus
which contradicts the fact that $\widetilde{\T}_a$ is anisotropic.
\end{proof}

\section{The correspondence between MAD subalgebras  and maximal split tori}
Throughout this section $R$ will denote an object of $\kalg$ such that $\XX = \Spec(R)$ 
is normal integral and noetherian and $K$ its fraction field. 
The purpose of this section is to establish the following fundamental correspondence.

\begin{theorem}\label{maincorrespondence}
Let $\GG$ be a semisimple simply connected  $R$--group and $\mathfrak{g}$ its Lie algebra.

\smallskip

\noindent {\rm (1)} Let $\mathfrak{m}$ be a {\rm MAD} subalgebra  of $\mathfrak{g}$. Then $Z_{\GG}(\mathfrak{m})$ is a reductive $R$--group 
 and its radical
contains a unique maximal split torus $\SS(\mathfrak{m})$
of ${\GG}$.

\smallskip

\noindent {\rm (2)} Let $\SS$ is a maximal split torus of $\GG,$ and let $\mathfrak{m}(\SS)$ be the unique maximal
subalgebra of  Lie algebra ${\rm Lie}\,(\SS)$ which is an {\rm AD} subalgebra of $\gg$ (see Remark  \ref{m(S)}). 
Then $\mathfrak{m}(\SS)$ is a {\rm MAD} subalgebra of $\gg.$

\smallskip

\noindent {\rm (3)} The process $\mathfrak{m}\to \SS(\mathfrak{m})$
and $\SS\to \mathfrak{m}(\SS)$ described above gives a bijection between
the set of {\rm MAD} subalgebras  of $\mathfrak{g}$ and the set of maximal
split tori of $\GG$.

\smallskip

\noindent{\rm (4)} If $\mathfrak{m}$ and $\mathfrak{m}'$ are two {\rm MAD} subalgebras  of $\gg,$ then for $\mathfrak{m}$ and $\mathfrak{m}'$ to be  conjugate under the adjoint action of 
$\GG(R)$ it is necessary and sufficient that the maximal split tori $\SS(\mathfrak{m})$ and  $\SS(\mathfrak{m}')$ 
be conjugate under the adjoint action of $\GG(R)$ on $\gg.$  
\end{theorem}

\begin{remark}\label{natureofm(S)} Since $\SS$ is split we have  ${\rm Lie}\,(\SS) = X({\SS})^{\rm o} \otimes_\Z R$
 where $X(\SS)^{\rm o}$ is the cocharacter group of $\SS$.
As we shall see in the proof of Lemma \ref{splittorus} $\mm(\SS) =  X({\SS})^{\rm o} \otimes_\Z k.$
\end{remark}

The proof of the Theorem will be given at the end of this section after a long list of 
preparatory results. What is remarkable about this correspondence is that MAD subalgebras 
exist over $k$ but not over $R$ while, in general, the exact opposite is true for split 
tori of $\GG.$ It is this correspondence that allows us to use the methods from  
\cite{SGA3} to the study of conjugacy questions.
\smallskip

We begin with some general observations and fixing some notation that will be used throughout the proofs of this section.
Since $\XX$ is connected all geometric fibers of $\GG$ are of the same type.  
Let $\dG$ be the corresponding Chevalley group over $k$ and $\dg$ its Lie algebra. 




\begin{lemma}\label{upperbound}
Let $\mathfrak{m}$ be an AD subalgebra of $\mathfrak{g}$.
Then

\smallskip
\noindent {\rm (1)}  ${\rm dim}_k(\mathfrak{m}) \leq
{\rm rank}\,(\dg).$
In particular any AD subalgebra of $\gg$ is included inside a MAD subalgebra of $\gg.$ 

\smallskip

\noindent {\rm (2)}   The natural map $\mm\otimes_k R \to R\,\mm$
is an $R$--module isomorphism.  In particular
$R\,\mathfrak{m}$ is a free $R$--module of rank $={\rm dim}_k(\mathfrak{m}).$

\smallskip

\noindent {\rm (3)} Let $\{v_1,\ldots, v_m\}$ be a $k$--basis
of $\mathfrak{m}.$ For every $x\in \XX$ the elements $v_i\otimes 1\in
\gg\otimes_R R_x$ are $R_x$--linearly independent. Similarly if we replace $R_x$ by K or any field extension of $K.$
\end{lemma}

\begin{proof}

 The three assertions are of local nature, so 
we can assume that $R$ is local.
We will establish the Lemma by first reducing the problem to the split case.
According to  Remark~\ref{localsplit}
there exists a finite \'etale  extension $\widetilde{R}/R$ 
such that  $\widetilde{R}$ is integral and normal and
 $\GG \times_R \widetilde{R} \simeq \dG_{\widetilde{R}}$.
Note that
the canonical map $\gg \to \gg \otimes_R \widetilde{R}\simeq  \dg \otimes_k \widetilde{R}$ 
is injective and that
if  $\{v_1,\ldots, v_m\}$ are $k$--linearly independent elements 
of $\mathfrak{m}$ which are $R$--linearly dependent, then the image of the elements  
$\{v_1,\ldots, v_m\}$ on 
${\rm Lie}(\GG \times_R \widetilde{R}) \simeq \dg \otimes_k \widetilde{R}$ 
are $k$--linearly independent
and are $\widetilde{R}$--linearly dependent.

Let $\widetilde{K}$ be the field of fractions of $\widetilde{R}.$ By 
Remark \ref{adup} 
the image of $\mm$ under the  injection $\gg \hookrightarrow \gg \otimes_R \widetilde{R} \simeq \dg\otimes_k \widetilde{R}$ is an AD subalgebra of $\dg\otimes_k \widetilde{R}$. By \cite[theo.1.(i)]{P1} the dimension of $\mm$ is at most the rank of $\dg.$ This establishes (1). 

 As for (2) and (3),  the crucial point--as explained
in \cite[prop. 4]{P1}--lies in the fact that the image $\widetilde{\mm}$ of $\mm$ under the  injection
$\gg \hookrightarrow \dg\otimes_k\widetilde{K}$
 sits inside a split Cartan subalgebra $\mathcal {H}$ 
of the split semisimple $\widetilde{K}$-algebra
$\dg\otimes_k \widetilde{K}$. 
 Consider the basis $\{\check{\omega}_1,\ldots,
\check{\omega}_\ell\}$ of $\mathcal{H}$ consisting
of the fundamental coweights  for a base $\alpha_1, \ldots,\alpha_\ell$ of the root system
of $(\dg \otimes_k \widetilde{K},\mathcal{H})$. Let $1 \leq n \leq m$ be such that $\{\widetilde{v}_1, \ldots \widetilde{v}_n\}$ is a maximal set of $\widetilde{K}$--linearly independent elements of $\dg (\widetilde{K})$. To establish (2) and (3) it will suffice to show that $n = m.$ 

Assume on the contrary that $n < m.$ Write  $\widetilde{v}_i =\sum c_{ji}\check{\omega}_j$ with
$c_{1i}, \ldots c_{\ell i}$ in $\widetilde{K}.$ The fact that the eigenvalues of
${\rm ad}_{\dg(\widetilde{K})}(\widetilde{v}_i)$ belong to $k$ show that the $c_{ji}$
necessarily belong to $k$. Indeed $\widetilde{v}_i$ acts on $\dg(\widetilde{K})_{\alpha_j}$ as multiplication by the scalar $c_{ji}.$

 Let $v = v_{n + 1}.$ Write $\widetilde{v} = \sum_{i=1}^n a_i \widetilde{v}_i$ with 
 $a_1, \ldots , a_n$ in $\widetilde{K}.$
Let $c_{jn+1}=\lambda_j$.
 Then $\langle\alpha_j, \widetilde{v}\rangle = \lambda_j$ and
 $$ \widetilde{v} = \sum_j(\sum _i a_ic_{ji})\check{\omega}_j = \sum_j \lambda_j\check{\omega}_j.$$
 Therefore for all $1 \leq j \leq \ell$ we have $\sum _i a_ic_{ji}  = \lambda_j.$
 
 Write $\widetilde{K} = k \oplus W$ as a $k$--space and use this decomposition to write $a_ i = d_i + w_i$ . Then $\sum _i d_ic_{ji}  = \lambda_j.$
  A straightforward calculation shows that $\langle\alpha_j , \widetilde{v} -\sum_i d_i\widetilde{v_i}
  \rangle = 0$ for all $j.$ 
  This forces 
  $$v_{n + 1} = v = \sum_i d_i v_i$$
   which contradicts the linear independence of the $v_i's$ over $k$.
\end{proof}

\begin{remark}\label{eigendecomposition}
Let $\SS < \GG $ be a split torus. Then there exist characters
$\lambda_i: \SS \to {\G}_{m,R}$ for $ 1\leq i\leq l$ such that
$
\gg=\oplus_{i=1}^l \, {\gg}_{\lambda_i}
$
where
$$
{\gg}_{\lambda_i}=\{\, v\in \gg:\ {\rm Ad}(g)v=\lambda_i(g)v \ \forall g\in
{\SS}(R)\,\}.
$$

At the Lie algebra level the situation is as follows. Let $\ss=
{\rm Lie}\,(\SS)\subset \gg$. Then $\ss\subset \SS(R[\varepsilon])$.
We avail ourselves of the useful convention that if $s\in \ss$ then to view
$s$ as an element of $\SS(R[\varepsilon])$ we write $e^{s\varepsilon}$.
There exist unique $R$--linear functionals $d\lambda_i:\ss \to
R$ such that $$
\lambda_i(e^{s\varepsilon}) =1+d\lambda_i(s)\varepsilon
\in R[\varepsilon]^{\times}=
\bG_{m,R}(R[\varepsilon]).
$$
Then for $s\in \ss$ and $v\in {\gg}_{\lambda_i}$ we have the following
equality in $\gg$
\begin{equation}\label{ad}
[s,v]=d \lambda_i(s)v.
\end{equation}
\end{remark}

\begin{lemma}\label{splittorus}
 Consider the restriction ${\rm Ad}_{\SS}: \SS \to {\rm \bf Gl}\,(\gg)$
of the adjoint representation of $\GG$ to $\SS$. There exists a finite
number of
characters $\lambda_1,\ldots,\lambda_l$ of $\SS$ such that
$
\gg=\bigoplus_{i=1}^l\,{\gg}_{\lambda_i}.
$
The $\lambda_i$ are unique and
$$
{\mm}(\SS)=\{\, s\in {\rm Lie}\,({\SS})\subset {\SS}\,(R[\varepsilon])\ : \
d\lambda_i(s)\in k\,\}.
$$
Furthermore $${\rm dim}_k \big({\mm}(\SS)\big)={\rm rank}(\SS)=
{\rm rank}_{R-mod}\big(R{\mm}(\SS)\big)
$$
and ${\rm Lie}(\SS)=R\mm(\SS)$.
\end{lemma}
\begin{proof} We appeal to the explanation given in
Remark~\ref{eigendecomposition}.
Let
$$
\nn=\{\,s\in \ss\ :\ d\lambda_i(s)\in k\ \forall i\,\}.
$$
Then~(\ref{ad}) shows not only that $\nn\subset \ss$ is an
AD subalgebra of $\gg$, but in fact that $\mm(\SS)\subset \nn$. By maximality
we have $\mm(\SS)=\nn$ as desired.

We now establish the last assertions. Let
$n$ be the rank of $\SS$, so $\SS\simeq {\G}_{m,R}^n$
and the character lattice $X(\SS)$ of $\SS$ is generated by the projections
$\pi_i: {\G}_{m,R}^n\to  {\G}_{m,R}$. Since the kernel of the adjoint representation 
of $\GG$ is finite
the sublattice of $X(\SS)$ generated by $\lambda_1,\ldots,\lambda_{\ell}$ has finite index;
in particular every character $\pi$ of $\SS$ can be written as a linear combination
$\pi=a_1\lambda_1+\cdots+a_{\ell}\lambda_{\ell}$ with rational coefficients $a_1,\ldots,
a_{\ell}$
and hence $d\pi=a_1d\lambda_1+\cdots+ a_nd\lambda_{\ell}$.
Similarly $\pi$ can be written as  $\pi=a_1\pi_1+\cdots +a_n\pi_n$ 
with $a_1,\ldots,a_n\in\mathbb{Z}$ and we then have
$d\pi=a_1d\pi_1+\cdots+ a_nd\pi_n$.
It follows that
$$
\begin{array}{ccl}
{\mm}(\SS) & = &\{\, s\in \ss\ : \
d\lambda_i(s)\in k\ \forall i\,\}\\
& = & \{\,s\in \ss \ : \ d \pi(s)\in k \ \forall \pi\in X(\SS)\,\}\\
& = & \{\,s\in \ss \ : \ d\pi_i(s)\in k\ \forall i\,\}.
\end{array} 
$$
The identification $\SS\simeq {\G}_{m,R}^n$ induces 
the identification $\ss\simeq {\G}_{a,R}^n$. The above equalities yield 
$$
{\mm}(\SS)\simeq \{\,(s_1,\ldots,s_n)\ : \ s_i\in k \ \forall i\,\},
$$
hence the last assertions follow immediately.
%
\end{proof}

\begin{proposition}\label{directsummand}
 Let $\mm$ be an {\rm AD} subalgebra of $\gg$.
Then the submodule $R\mm$ is a direct summand of $\mathfrak{g}$.
\end{proposition}
\begin{proof} Let $M=\mathfrak{g}/R\mm.$ 
Assume for a moment that $M$ is a projective $R$--module.
Then the exact sequence
$$
0\longrightarrow R\mm\longrightarrow \mathfrak{g}\longrightarrow
M\longrightarrow 0
$$
is split and the Proposition follows.

Thus it remains to show that $M$ is a projective $R$--module
or, equivalently, that for every prime ideal $x$ of 
$R$ the localized $R_x$--module $M_x$ is free. Since 
localization is a left exact functor, and by Lemma \ref{upperbound} we 
have $(R\mathfrak{m})_x =  R_x\mathfrak{m}$ the sequence
$$
0\longrightarrow R_x\mm\longrightarrow \mathfrak{g}_{R_x}\longrightarrow
M_x\longrightarrow 0
$$
is exact.
By Lemma~\ref{upperbound}(3), 
the elements $$v_1\otimes 1,\ldots, v_m \otimes 1\in
R_{x}\mathfrak{m} \subset \gg \otimes_R R_{x} = \gg_{x}$$ and the module $\gg_{x}$
satisfy the variation of Nakayama's lemma stated in \cite[cor. 1.8]{Lam}.
Hence $R_{x}\mm$ is a direct summand of $\mathfrak{g}$ and this implies that
$M_{x}$ is free.
\end{proof}

\begin{proposition}\label{fibers}  Let $\mm$ be an AD subalgebra of $\gg.$ Then 
%
$Z_\GG(\mm)$ is an affine $R$--group whose geometric fibres  are (connected) reductive groups. 
%
%
%
\end{proposition}
\begin{proof}
By Proposition \ref{directsummand} $R\mm$ is a direct summand of $\gg.$ It follows from 
\cite[II prop.1.4]{DG}  that $Z_{\GG}(R\mm) = Z_{\GG}(\mm)$ is a closed subgroup of $\GG$. In particular, $Z_{\GG}(\mm)$ is an affine scheme which is of finite type over $\Spec(R).$

Let $x\in {\rm Spec}(R)$ be a point and let 
$k(\ol{x})$ be an algebraic closure of $k(x).$ Since the functor 
$Z_{\GG}(\mm) =Z_{\GG}(R\mm)$ commutes
with base change, to verify the nature of its geometric fibers  
$Z_{\GG}(\mm)(\ol{x})$ we may look at 
$$
Z_{\GG}(R\mm)\otimes_R k(\ol{x})=
Z_{\GG(\ol{x})}(k(\ol{x})\,\mm(\ol{x}))$$
where $\GG(\ol{x})=\GG\otimes_R k(\ol{x})$ and $\mm(\ol{x})$ is the image of $\mm$ under
$\gg \to \gg \otimes_R k(\ol{x})$.
 Thus we  may assume without loss of generality that the ground ring is a field. By results of Steinberg (\cite[3.3 and 3.8]{[St75]}
 and  \cite[0.2]{[St75]}) we conclude that $Z_{\GG}(\mm)(\ol{x})$ is connected and reductive.  
 \end{proof}

\subsection{Flatness of $Z_{\GG}(\mm)$.}
Fix a split Cartan subalgebra $\dh$ of $\dg$. With respect to the adjoint
representation ${\rm ad}: \,\dg\to {\rm End}_k\,(\dg)$
we have the weight space decomposition  
$$\dg=\oplus_{\alpha\in \Sigma} \,\dg_{\alpha}$$ where $\alpha:\dh \to k$ is a linear function  such that the corresponding eigenspace $\dg_\alpha$ is non-zero. The kernel of the adjoint representation of $\dg$ is trivial, ${\rm dim}\,\dg_{\alpha}=1$
if $\alpha\not=0$ and $\dg_0=\dh$.
\begin{lemma}\label{centralizer2}
Let $\da \subset \dh$ be a subalgebra. Then:

{\rm (1)}  The centralizer $Z_{\dg}(\da)$ is a reductive
Lie algebra whose centre is contained in $\dh$.

{\rm (2)} If $a\in \da$ is in  generic position then
$Z_{\dg}(\da)=Z_{\dg}(a)$.
\end{lemma}
\begin{proof} 
(1) The centralizer of $\da$ is generated by $\dh$ and those
$\dg_{\alpha}$ for which $\alpha(x)=0$ for every $x\in \da$. It is a well-known fact that this algebra is reductive.

(2) The inclusion $\subset$ is obvious. Conversely, the centralizer of $a$ is generated
by $\dh$ and those $\dg_{\alpha}$ for which $\alpha(a)=0$.
Since $a$ is generic all such roots $\alpha$ also satisfy $\alpha(x)=0$ for all $x\in\da$.
\end{proof}

\begin{lemma}\label{uniqueness} Let $a_{\alpha}\in k,\ \alpha\in \Sigma$. 
Then there exists at most one element $h\in \dh$
such that $\alpha(h)=a_{\alpha}$.
\end{lemma}
\begin{proof} Since the kernel of the adjoint representation
of $\dg$ is trivial the result follows.
\end{proof}
\begin{lemma}\label{Srationality}
Let $S$ be an object of $\kalg$. Let $v\in\dh \otimes_k S$ be an ad $k$--diagonalizable element of $\dg \otimes_k S.$ If $S$ is an integral domain then
$v\in \dh$.
\end{lemma}
\begin{proof} Let $F$ be a field of quotients of $S$  and view $v$ as an element of  $\dg \otimes_k F.$ The eigenvalues $a_{\alpha}$ of $v$ with respect to the adjoint representation
are $a_{\alpha}=\alpha(v)$. By assumption they all belong to k. Thus the nonhomogeneous linear system 
$\alpha(x)=a_{\alpha}$, $\alpha\in \Sigma$, has a solution over $F$, namely $v$. Since the coefficients of this system of equations are in $k$
it also has a solution over $k$ [see the proof of Lemma \ref{upperbound}(2)]. By Lemma~\ref{uniqueness} such a solution is unique, hence
$v\in\dh$.
\end{proof}

Fix an arbitrary element $h\in \dh$. Recall that $\dG$ acts
on $\dg$ by conjugation and it is known that the orbit $\mathcal{O}_h=\dG\cdot h$
is a Zariski closed subset of $\dg$ (because $h$ is semisimple). Let $\bL\subset \dG$ 
be the isotropy subgroup of $h$ in $\dG$. 
As we saw above $\bL$
is a reductive subgroup and we have an exact sequence
$$
1\longrightarrow \bL\longrightarrow \dG\stackrel{\phi}{\longrightarrow} \dG/\bL\longrightarrow 1.
$$


The algebraic $k$--varieties $\mathcal{O}_h$ and $\dG/\bL$ have the distinguished points $h$
and  the coset $e=1\cdot \bL$ respectively.
The group $\dG$ acts on both $\mathcal{O}_h$ and $\dG/\bL$ in a natural way and 
there exists a natural
$\dG$-equivariant isomorphism $\lambda: \mathcal{O}_h \simeq \dG/\bL$ which takes $h$ 
into $e$ (see \cite[III \S9]{Bor} for details).
Hence if $R$ is an object in $\kalg$ and $x\in \mathcal{O}_h(R),$ then $x$ and $h$
are conjugate by an element in $\dG(R)$ if and only if
$\lambda(x)\in \dG(R)\cdot e$. 
\smallskip

We now return to our simply connected semisimple $R$--group $\GG$
and its Lie algebra $\gg$. 

\begin{lemma}\label{ADflatness}
Let $\mm$ be an AD subalgebra of $\gg.$ The affine scheme $Z_{\GG}(\mm)$  is flat over ${\rm Spec}(R)$.
\end{lemma}

\begin{proof} That $Z_{\GG}(\mm)$ is an affine scheme over $R$ has already been established. 
Since flatness is a local property it will suffice to establish the result after we replace $R$ by its localization at each element of $\XX.$ Lemma \ref{normalsplit} provides
a finite \'etale connected cover $\widetilde{R}/R$ which splits $\GG$. 
By replacing $R$ by $\widetilde{R}$ we reduce the problem to the split case. Summarizing, 
without loss of generality we may assume that  $\GG=\dG \times_k R$,
$\gg= \dg \otimes_k R := \dg_R$ and $R$ is a local domain. 

As observed in Lemma~\ref{upperbound}   $\mm$ is contained in a split 
Cartan subalgebra $\mathcal{H}$ 
of $\dg \otimes_k K := \dg_K$. Fix a generic vector $v\in \mm\subset \dg_K$.
Let $\{\,a_{\alpha},\,\alpha\in\Sigma\,\}$ be the family of all eigenvalues of $v$ with respect to
the adjoint representation of $\dg_K$. Since $\mm$ is  an AD subalgebra of $\dg_R$,
 we have $a_{\alpha}\in k$ for every
$\alpha\in\Sigma$. 

\begin{sublemma}\label{reduction-to-h}  
There exists a unique vector $h\in\dh$ whose eigenvalues with respect to
the adjoint representation are $\{\,a_{\alpha},\,\alpha\in\Sigma\,\}$. Moreover if $v$ and $h$ 
are viewed as elements of $\dg_K$, then they are conjugate under $\dG(K).$
\end{sublemma}

\begin{proof} Uniqueness follows from Lemma~\ref{uniqueness}. As for existence, we note 
that $\mathcal{H}$ and ${\dh}_K$
are conjugate over $K$, hence ${\dh}_K$ clearly contains an element
with the prescribed property. 
By Lemma~\ref{Srationality} this element is contained in $\dh$. The conjugacy assertion 
follows from the construction  of $h.$
\end{proof}

We now come back to  the $\dG$-orbit $\mathcal{O}_h$ of $h$.
We remind  the reader that this  is a closed subvariety of $\dg$.

\begin{sublemma}\label{R-points} $v\in \mathcal{O}_h(R)$.
\end{sublemma}
\begin{proof} The element $v\in \dg_R$ can be viewed as a morphism
$$\phi_v:{\rm Spec}\,(R)\to \dg.$$  
The image of the generic point ${\rm Spec}(K)
\to {\rm Spec}(R)\to \dg$ is contained in $\mathcal{O}_h$ for 
$v$ and $h$ are conjugate over $K$.
 Since $\mathcal{O}_h$ is a closed subvariety of $\dg$ and since ${\rm Spec}(R)$
 is irreducible  it follows 
 that $\phi_v$ factors through the embedding
 $\mathcal{O}_h\hookrightarrow  \mathfrak{g}$.
\end{proof}

 To finish the proof of 
Lemma \ref{ADflatness} we first consider the particular case when $\mm$
 is contained in $\dh$. Then $Z_{\GG}(\mm)$ is obtained from the variety $Z_{\dG}(\mm)$ by the base
 change $R/k$ so that flatness is clear. 
  
In the general case, let $h\in\dh$ be the element provided by
Sublemma~\ref{reduction-to-h}.
By Sublemma~\ref{R-points} we have $v\in\mathcal{O}_h(R)= (\dG/\bL)(R)$.
Denote by $R^{sh}$ the strict henselisation of the local ring $R$, that is the
simply connected cover of $R$ attached to a separable closure $K_s$ of $K$
(see  \cite[\S X.2 ]{Ra}).
Since the map $p: \dG \to \dG/\bL$ is smooth and surjective,  Hensel's lemma \cite[\S 4]{M1}
shows that    $\dG(R^{sh}) \to (\dG/\bL)(R^{sh})$ is surjective.
But $R^{sh}$ is the inductive limit of the finite (connected)  Galois  covers of $R$, 
so there exists one such  cover $R'$ and a point $g' \in \dG(R')$ such that 
$v= g'.h$. Up to replacing $R$ by $R'$ (which is a noetherian normal domain) 
we may assume that $v=h$. 
 
We now recall  that $Z_{{\dg}_R}(h)=Z_{{\dg}_R}(\mm)$ since  $h=v\in \mm$ is a generic vector.
Since the center of $Z_{{\dg}_R}(h)$ is contained in $\dh_R$ and since
$\mm$ is contained in the center of its centralizer we have $\mm\subset \dh_R$.
Applying Lemma~\ref{Srationality} then shows that $\mm\subset \dh$.
Thus we have reduced the general case to the previous one.
\end{proof}

 
 \begin{proposition}\label{flatreductive} If $\mm$ is an AD subalgebra of $\gg$ then $Z_{\GG}(\mm)$ 
 is a reductive $R$--group. 
 \end{proposition}

 \begin{proof} Since   $Z_{\GG}(\mm)$ is flat and also finitely presented over $R$  the 
 differential criteria for smoothness shows that  $Z_{\GG}(\mm)$ is in fact smooth over 
 $R$ because of 
Proposition \ref{fibers}.
 Furthermore, geometric fibers of $Z_{\GG}(\mm)$ are (connected) 
 reductive groups in the usual sense (this last again by Proposition \ref{fibers}). 
 By definition $Z_{\GG}(\mm)$ is a reductive $R$--group.
 \end{proof}


\noindent{\it Proof of Theorem \ref{maincorrespondence}.} 
(1) Let  $\mm$ be a MAD subalgebra of $\gg,$ and let $\SS$ denote the maximal 
split torus of the radical 
$\TT$ of the reductive $R$-group $Z_\GG(\mm).$ By Remark~\ref{m(S)} 
the Lie algebra of $\SS$ contains a 
unique maximal subalgebra $\mm(\SS)$ which is an AD--subalgebra of $\gg$.
By definition
$\SS < \HH=Z_{\GG}(R\mm)$. Denote  ${\rm Lie}(\SS)$ by $\ss.$ 
Since $\ss \subset \SS(R[\varepsilon])$ it follows that
in $\gg$ we have $[ \ss,R\mm]=0.$ In particular since
$\mm\,(\SS)\subset \ss$ we have
$[\,\mm(\SS),\mm\,]=0.$ But then by Remark \ref{m(S)}  $\mm+\mm(\SS)$ is an 
AD subalgebra of $\gg$.
Since $\mm$ is a MAD subalgebra we necessarily have $\mm(\SS)\subset \mm$ and now we 
are going to show
that $\mm(\SS)=\mm$.

Recall that $K$ denotes the quotient field of $R.$  By Lemma~\ref{splittorus} we have 
${\rm dim}\,(\mm(\SS))=
{\rm rank}\,(\SS)$, so that to establish that $\mm(\SS) = \mm$ it will suffice to 
show that ${\rm rank}\,(\SS)
\geq {\rm dim}_k(\mm),$ or  equivalently that ${\rm dim}_K({\SS}_K)
\geq {\rm dim}_k(\mm)$. 

We have ${\HH}_K=Z_{{\GG}_K}(R\mm) = Z_{{\GG}_K}(K\mm)$, as can be seen from 
the fact that the computation of the centralizer commutes with base change.
Since $\SS$ is the maximal split
torus  of $\TT$ then ${\SS}_K$ is the maximal split torus of
${\TT}_K={\rm rad}\,({\HH}_K)$ by Lemma \ref{maxtor}.
We also have
$${\rm Lie}\,({\HH}_K)={\rm Lie}\,(Z_{{\GG}_K}(R\mm))=
{\rm Lie}\,(Z_{{\GG}_K}(K\mm))
=Z_{{\gg}_K}(K\mm).$$ 
Since $K\mm$ is in the centre of $Z_{{\gg}_K}(K\mm)={\rm Lie}\,({\HH}_K)$
and the centre of ${\rm Lie}\,({\HH}_K)$ coincides with
${\rm Lie}\,({\TT}_K)$ 
we conclude that $K\mm\subset
{\rm Lie}\,({\TT}_K)$. On the other hand $K\mm$ is an AD subalgebra of
${\gg}_K$, so that by Lemma~\ref{splitradical} $K\mm\subset {\rm Lie}\,({\SS}_K)$.
This shows that ${\rm dim}_K(K\mm)\leq {\rm dim}_K({\SS}_K)$.
But by
Lemma~\ref{upperbound}(3) we have ${\rm dim}_k(\mm)={\rm dim}_K(K\mm)$. This
completes the proof that $\mm(\SS)=\mm$.

Now it is easy to finish the proof that $\SS$ is a maximal split torus in $\GG$. 
If $\SS$ is contained in a split torus $\SS'$ of larger rank then $\mm(\SS)\subset \mm(\SS')$ 
is a proper
subalgebra which contradicts to the fact that $\mm=\mm(\SS)$ is a MAD subalgebra.


(2) Let $\SS$ be a maximal split torus
of $\GG$, and let $\ss={\rm Lie}\,(\SS)$ be its Lie algebra.
By Remark~\ref{m(S)} $\ss$ contains a unique maximal subalgebra
$\mm(\SS)=\mm$ which is an AD--subalgebra of $\gg$.  We have by Lemma~\ref{splittorus} that $R\mm=
{\rm Lie}\,(\SS)$. Thus, appealing to Proposition \ref{centralizer} and  Lemma \ref{upperbound}(1) we obtain
$$
Z_{\GG}(\mm) = 
Z_{\GG}(R\mm)=Z_{\GG}(\ss)=Z_{\GG}(\SS).
$$
 We claim
that $\mm$ is maximal. 
Assume otherwise. Then by Lemma \ref{upperbound}(1) $\mm$ is properly included
in a MAD subalgebra ${\mm}'$ of $\gg$. We have
$$
{\HH}':=Z_{\GG}(R\,{\mm}')\subset
\HH:=Z_{\GG}(R\mm)= Z_{\GG}(\SS).
$$
By Proposition \ref{flatreductive} $\HH'$ and $\HH$ are  reductive $R$--groups. Let $\TT'$ and $\TT$ be their radicals
and let $\TT'_d$, $\TT_d$ be their maximal split tori.
We  have $\SS\subset \TT\subset \TT'$ and hence $\SS\subset \TT_d\subset \TT'_d$. But
$\SS$ is a maximal split torus in $\GG$. Therefore $\SS=\TT'_d=\TT_d$ and this implies
$\mm=\mm(\SS)=\mm(\TT_d)=\mm(\TT'_d)$. Recall that in part (1) we showed 
that $\mm(\TT'_d)=\mm'$ and thus $\mm=\mm'$ -- a contradiction.

(3) If $\mm$ is a MAD subalgebra of $\gg,$ the corresponding maximal split
torus $\SS(\mm)$ is the maximal split torus of the radical of ${\HH}=Z_{\GG}(R\mm)$.
The  proof of (1) shows that the MAD subalgebra  corresponding to $\SS(\mm)$
is 
$\mm$.

Conversely, if $\SS$ is a maximal split torus of $\GG$
then the maximal split torus corresponding to $\mm(\SS)$
is the maximal split torus of the radical of the reductive group
$Z_{\GG}(R\mm(\SS))=Z_{\GG}(\ss)=Z_{\GG}(\SS)$ as explained in the proof
of (1). Clearly $\SS$ is inside the radical of $Z_{\GG}(\SS)$.
Since $\SS$ is maximal split in $\GG$ it is maximal split in the radical
of $Z_{\GG}(\SS)$. Thus $\SS={\SS}'$. 

(4) Follows from the construction and functoriality in the definition of the adjoint action at the Lie algebra and group level. \qed

\section{A sufficient condition for conjugacy}

In this section $R$ denotes a normal noetherian domain and $K$ its field of quotients.
Let $\GG$ be a reductive group scheme over
$R$. We say that a maximal split torus $\SS$ of $\GG$ is {\it generically maximal split} if
${\SS}_K$ is a maximal split torus of ${\GG}_K.$ 

\begin{proposition}\label{obstacle}
Let $\SS$ be a generically maximal
split torus of $\GG$. If
\begin{equation}\label{Hzar}
H^1_{{\it Zar}}\big(R, Z_{\GG}(\SS)\big)=1
\end{equation}
then all generically maximal split tori
of $\GG$ are conjugate under $\GG(R)$.
\end{proposition}
 We begin with two preliminary results.

\begin{lemma}\label{frankfurt} Let ${\WW}$ be a finite \'etale $R$-group with $R$ normal. Let $K$ be the field of quotients of $R.$
Then

\smallskip

\noindent {\rm (1)} The canonical map
$$
\chi: H_{\et}^1(R,{\WW}) \longrightarrow H^1(K,{\WW}_K)$$
is injective.

\smallskip

\noindent {\rm (2)} $H^1_{Zar}(R,{\WW})=1$.

\end{lemma}
\begin{proof} (1) Because of the assumptions on $\WW$ we can
  compute $H_{\et}^1(R, \WW)$ as the
 limit of $H_{\et}^1(S/R, \WW)$ with $S$ a connected finite Galois
 extension of $R.$ 
 
 Let $\Gamma = \Gal (S/R).$ It is well-known that $\WW$ corresponds to a finite group $W$ together with an action of the algebraic fundamental group of $R$, and that $H_{\et}^1(S/R, \WW) = H^1\big(\Gamma , \WW(S)\big)$
(see \cite[XI \S5]{SGA1}).  
 If $L$ denotes the field of quotients of $S$ then $L/K$ is also Galois
 with Galois group naturally isomorphic to $\Gamma$  as explained in  \cite[Ch.5 \S2.2 theo. 2]{Bbk}. 
Our map $\chi$ 
is obtained by the base change $K/R.$  By the above considerations  the problem 
reduces to the study of the map
$$\chi: H^1\big(\Gamma ,{\WW}(S)\big) \longrightarrow H^1\big(\Gamma ,{\WW}(S\otimes_RK)\big)$$
when passing to the limit over $S.$  Since $R$ is normal by 
\cite[18.10.8 and 18.10.9]{EGAIV} we have $S \otimes_RK = L.$ 
If $S$ is sufficiently large, $\WW(S) = W = \WW(L)$. The compatibility of the two Galois actions gives the desired injectivity. 

(2) It is clear that $H^1_{Zar}(R,{\WW})$ is in the kernel of $\chi.$
\end{proof}

\begin{lemma}\label{cocycle} Let $\SS$ and ${\SS}'$ be generically 
maximal split tori of $\GG$. Then
%
%
the transporter $\tau_{\SS,{\SS}'}={\rm \bf Trans}_{\GG}(\SS,{\SS}')$
is a (Zariski) locally trivial $N_{\GG}(\SS)$--torsor over $R.$
\end{lemma}
\begin{proof}

By \cite[XI, 6.11 (a)]{SGA3}, $\tau_{\SS,{\SS}'}$ is a closed subscheme
of $\GG$. It is clearly a right  (formal) torsor under the affine $R$--group
$N_{\GG}(\SS)$. Since ${\SS}_{R_\mathfrak{p}}$ and ${\SS}'_{R_\mathfrak{p}}$ are maximal
split tori of ${\GG}_{R_\mathfrak{p}}$ they are conjugate under $\GG(R_\mathfrak{p})$
by~\cite[XXVI, 6.16]{SGA3}.
Thus $\tau_{\SS,{\SS}'}$ is an $N_{\GG}(\SS)$--torsor which is locally trivial
(i.e. there exists a Zariski open cover $\XX=\cup\,\XX_i$ such that
$\tau_{\SS,{\SS}'}(\XX_i)\not=\emptyset).$
\end{proof}

\noindent {\it Proof of Proposition \ref{obstacle}} Let ${\SS}'$
be a generically  maximal split torus of $\GG$. The transporter
$\tau_{\SS,{\SS}'}$ yields according to Lemma~\ref{cocycle} an element
$\alpha\in H^1_{Zar}(R, N_{\GG}(\SS))$. Our aim
is to show that $\alpha$ is trivial.

Consider the exact sequence (on $\XX_{\acute{e}t}$) of $R$--groups
$$
1\longrightarrow Z_{\GG}(\SS) \longrightarrow N_{\GG}(\SS)
\longrightarrow \WW\longrightarrow 1
$$
with $\WW=N_{\GG}(\SS)/Z_{\GG}(\SS)$. Then $\WW$ is a finite
\'etale group over $R$ (see \cite[XI, 5.9]{SGA3}). By Lemma  \ref{frankfurt}(2)
the image of $\alpha$ in $H_{\et}^1(R,\WW),$ which we know lies in $H^1_{Zar}(R,\WW),$
is trivial. Thus we may assume $\alpha\in H_{\et}^1\big(R,Z_{\GG}(\SS)\big)$.
To finish the proof we need to show that
$$
\alpha\in {\rm Im}\,\big[H^1_{Zar}\big(R,Z_{\GG}(\SS)\big)\longrightarrow
H_{\et}^1\big(R,Z_{\GG}(\SS)\big)\big].
$$
For this it suffices to show that the image $\alpha_\mathfrak{p}$ of
$\alpha$  in
$$H_{\et}^1\big(R_\mathfrak{p},Z_{\GG}(\SS)\times_R R_\mathfrak{p}\big)=
H_{\et}^1\big(R_\mathfrak{p},Z_{{\GG}_{R_\mathfrak{p}}}({\SS}_{R_\mathfrak{p}})\big)$$ is trivial for all $\mathfrak{p}\in \XX.$

Since $\SS$ is generically maximal split, ${\SS}_{R_\mathfrak{p}}$ is a maximal
split torus of ${\GG}_{R_\mathfrak{p}}$. Similarly for ${\SS}'_{R_\mathfrak{p}}$. Now
by~\cite[XXVI prop. 6.16]{SGA3} ${\SS}_{R_\mathfrak{p}}$ and ${\SS}'_{R_\mathfrak{p}}$ are conjugate
under ${\GG}_{R_\mathfrak{p}}(R_\mathfrak{p}) = {\GG}(R_\mathfrak{p})$, Thus the image of $\alpha$ under the composition of
the natural maps
$$
H_{\et}^1\big(R,Z_{\GG}(\SS)\big)\to
H_{\et}^1\big(R,N_{\GG}(\SS)\big)\to
H_{\et}^1\big(R_{\mathfrak{p}}, N_{{\GG}_{R_\mathfrak{p}}}({\SS}_{R_\mathfrak{p}})\big)
\to
H_{\et}^1(R_\mathfrak{p},{\GG}_{R_\mathfrak{p}})
$$
is trivial.
Let $\PP$ be a parabolic subgroup of ${\GG}_{R_\mathfrak{p}}$ containing
$Z_{{\GG}_{R_\mathfrak{p}}}({\SS}_{R_\mathfrak{p}})$ as a Levi subgroup 
(see Lemma \ref{facts}).
Then (see the proof of \cite[XXVI cor. 5.10]{SGA3}) we have
$$
H_{\et}^1(R_\mathfrak{p},Z_{{\GG}_{R_\mathfrak{p}}}({\SS}_{R_\mathfrak{p}})) \simeq H_{\et}^1(R_\mathfrak{p},\PP) \hookrightarrow
H_{\et}^1(R_\mathfrak{p},{\GG}_{R_\mathfrak{p}}).
$$
It now follows that $\alpha_\mathfrak{p}$ is trivial. \qed

\subsection{A counter-example to conjugacy for multiloop algebras}\label{counterexample}

Let $\GG$ and $\gg$ be as in Theorem \ref{maincorrespondence}. We know that the conjugacy
of two MAD subalgebras in $\mathfrak{g}$
is equivalent to the conjugacy of the corresponding maximal split tori.
The following example shows that in general maximal split
tori are not necessarily conjugate.

Let $D$ be the quaternion algebra over $R = R_2 =k[t_1^{\pm 1},t_2^{\pm 1}]$
with generators $T_1,T_2$ and relations
$T_1^2=t_1,\ T_2^2=t_2$ and $T_2T_1=-T_1T_2$ and let $A=M_2(D)$.
We may view $A$ as the $D$-endomorphism algebra of the free right rank $2$ module $V=D\oplus D$ over $D.$ 
Let $\GG={\rm \bf SL}\,(1,A)$. This is a simple simply connected $R$-group of absolute type ${\rm \bf SL}_{4,R}.$ It contains a split torus $\SS$ whose
$R$--points are matrices
of the form
$$
\left(
\begin{array}{cc}
x & 0\\
0 & x^{-1}
\end{array}
\right)
$$
where $x\in R^{\times}.$ It is well-known that this is a maximal split torus of $\GG.$.

Consider now the $D$-linear map $f:V=D\oplus D \to D$ given by
$$(u,v)\to (1+T_1)u-(1+T_2)v.$$ Let $\mathcal{L}$ be its kernel.
It is shown in \cite{GP1} that $f$ splits and that $\mathcal{L}$ is a projective $D$--module of rank $1$
which is not free. Since $f$ is split, we have another decomposition
$V\simeq \mathcal{L}\oplus D$.
Let $\SS'$ be the split torus of $\GG$ whose $R$--points consist of linear transformations
acting on the first summand $\mathcal{L}$ by multiplication $x\in R^{\times}$
and on the second summand by $x^{-1}.$ As before, $\SS'$ is also a maximal split
torus of $\GG$.

We claim that $\SS$ and $\SS'$ are not conjugate under $\GG(R).$ To see this we note that
given $\SS$ we can restore the two summands in the decomposition $V=D\oplus D$
as  eigenspaces of elements $\SS(R)$.
Similarly, we can uniquely restore the two summands in the decomposition
$V=\mathcal{L}\oplus D$ out of $\SS'$. Assuming now that $\SS$ and $\SS'$ are conjugate
by an element in $\GG(R)$ we obtained immediately that
the $D$--submodule $\mathcal{L}$ in $V$ is isomorphic to one of the components
of $V=D\oplus D$, in particular $\mathcal{L}$ is free
-- a contradiction.

\section{The nullity one case}

In this section we look in detail at the case
 $R=k[t^{\pm 1}]$ where $k$ is assumed to be {\it algebraically closed}.
It is known that twisted forms of $\dg\otimes_kR$ are nothing but the derived
algebras of the affine Kac-Moody Lie algebras
modulo their centres~\cite{P2}.
We maintain all of our previous notation, except for
the fact that now we specify that $n=1$.
\begin{lemma} Every maximal split torus of $\GG$ is
generically maximal split.
\end{lemma}
\begin{proof}
Let $\SS$ be a maximal split torus of our simply connected
$R$--group $\GG$. We must show that $\SS_K$ is a maximal split torus of
$\GG_K$.
We consider the reductive $R$-group $\HH=Z_{\GG}(\SS)$, its derived
(semisimple) group $\mathcal{D}\,(\HH)$ which we denote by
$\HH'$, and the radical ${\rm rad}\,(\HH)$ of $\HH$.
Recall that ${\rm rad}\,(\HH)$ is a central torus of $\HH$ and that
we have an exact sequence of $R$--groups
$$
1\longrightarrow \bmu
\longrightarrow {\rm rad}\,(\HH)\times_R \HH'
\stackrel{m}{\longrightarrow} \HH \longrightarrow 1
$$
where $m$ is the multiplication 
and
$\bmu$ is a finite group of multiplicative type.


Since $\SS$ is central in $\HH$ it lies inside
${\rm rad}\,(\HH)$, hence it is the maximal split torus of ${\rad}\,(\HH)$.
Recall that by Lemma~\ref{maxtor}, $\SS_K$ is still the maximal split torus of
${\rad}\,(\HH)_K$. If $\SS_K$ is not a maximal
split torus of $\GG_K$, there exists a split torus $\SS'$
of  $\HH_K$ such that $\SS'$ is not a subgroup of ${\rm rad}\,(\HH_K)$.
Thus if we set $(\SS'\cap \HH'_K)^{\circ}=\TT$ then $\TT$ is a non-trivial split torus
of $\HH'_K$.
Then $Z_{\HH'_K}(\TT)$ is a Levi subgroup of a proper parabolic
subgroup $\PP$ of $\HH'_K$. 

Let $\t={\rm type}\,(\PP)$ be the type
of $\PP.$ 
Let ${\rm \bf Par}_{\t}(\HH')$ be the $R$-scheme of parabolic
subgroups of $\HH'$ of type $\t.$ Then ${\rm \bf Par}_{\t}(\HH')(K)
\not=\emptyset$. Since ${\rm \bf Par}_{\t}(\HH')$ is proper and $R$ is regular of
dimension $1$, it follows that ${\rm \bf Par}_{\t}(\HH')(R)\not=\emptyset$.
Let $\PP'$ be a parabolic subgroup $\HH'$ 
of type $\t$. 
It is a proper subgroup, so that by Proposition~\ref{G_m} $\PP'$
contains a copy of $\bG_{m,R}$. But then $m: \SS\times \bG_{m,R}
\to \HH$ yields a split torus of $\HH$ that properly contains
$\SS$ (since the multiplication map has finite kernel), which
contradicts the maximality of $\SS$.
\end{proof}
\begin{theorem} In nullity one all MAD subalgebras of $\gg$ are conjugate
under the adjoint action of $\GG(R)$.
\end{theorem}
\begin{proof}
 In view of the last Lemma and Proposition~\ref{obstacle}
it will suffice to show that if $\SS$ is a maximal split torus
of $\GG$, then $H^1_{Zar}(R,Z_{\GG}(\SS))=1$.
Since $Z_{\GG}(\SS)$ is a reductive $R$-group one in fact has
a much stronger result, namely that $H_{\et}^1(R,Z_{\GG}(\SS))=1$
(see~\cite[theo. 3.1]{P2}).
\end{proof}
\begin{remark} Let $G$ be the ``simply connected''
Kac-Moody (abstract) group corresponding
to $\gg$ (see [PK], and also \cite{Kmr} and  [MP] for details). We have the
adjoint representation ${\rm Ad}: G \to {\rm Aut}_{k-{\rm Lie}}(\gg)$.
The celebrated Peterson-Kac conjugacy theorem \cite{PK} for symmetrizable Kac-Moody
(applied to the affine case) asserts
that all MAD subalgebras of $\gg$ are conjugate under the adjoint action
of the group ${\rm Ad}\,(G)$ on $\gg$, while
our result gives conjugacy under the image of $\GG(R),$ where the image is that of 
the adjoint representation ${\rm Ad}: \GG \to
{\rm Aut}\,(\gg)$ evaluated at $R$.  In the untwisted case it is known that the two groups induce the same group of automorphisms of $\gg$ (see for example \cite{Kmr}). The twisted case appears to remain unstudied.
\end{remark}

\section{A density property for points of loop groups}
In this section $\XX = \Spec(R_n).$ For a description of $\pi_1(\XX, a)$ see \ref{aflaurent}.

Let $\dG$ be a linear algebraic $k$--group. Let $\eta \in Z^1\big(\pi_1(\XX,a), \bG(\ol{k})\big)$ be a loop cocycle and recall the decomposition  $\eta =(\eta^{geo}, z)$ into geometric and arithmetic parts described in Lemma \ref{Theta}. 
Recall that we may view $\eta^{geo}$ as a $k$--group homomorphism $
_{\infty}\bmu \to  {_z\bG}.$
We denote below by $({_z\bG)}^{\eta^{geo}}$ the centralizer
in ${_z\bG}$
 of  the  group homomorphism $\eta^{geo}$ . Thus defined $({_z\bG)}^{\eta^{geo}}$ is a $k$--subgroup of ${_z\bG}$

\begin{remark}\label{finitelevel} By continuity there exists $m$ and a Galois extension $\tilde{k}$ of $k$ such that $\eta$ factors through
$$
\eta:  \widetilde \Gamma_{n,m} \to {\bG}(\tilde k)
$$
where
$$\widetilde \Gamma_{n,m} := \Gal(
{R_{n,m}}\otimes_k \tilde{k}/R_n)=
 \bmu_m^n( \tilde k) \rtimes \Gal(\widetilde{k}/k)$$ where
$m > 0 $ and $\widetilde{k}/k$ is a finite Galois
extension containing
all $m$--roots of unity in $\overline {k}.$ By means of this interpretation $\eta$
can be viewed as a Galois cocycle in  
 $Z^1\big(\widetilde{\Gamma}_{n,m}, \bG(R_{n,m} \otimes_k \tilde{k})\big).$  We call this procedure ``reasoning at the finite level".

\end{remark}


We say that an  abstract group $M$ is {\it pro-solvable}
if it admits a filtration
$$
\cdots \subset M_{n+1} \subset M_n \subset  \cdots  \subset M_0 =M
$$
by normal subgroups such that $\cap\, M_n=1$ and $M_n/M_{n+1}$ is abelian
for all $n \geq 0$. If there exists a filtration such that $M_n/M_{n+1}$ are  $k$-vector spaces,
we say that $M$ is {\it pro-solvable in $k$-vector spaces}.

\begin{theorem} \label{density}
Let $\bG$ be a linear algebraic $k$-group such that $\bG^{\circ}$
is reductive. Let $\eta \in Z^1\big(\pi_1(\XX,a), \bG(\ol{k})\big)$ be a loop cocycle such that the twisted $R_n$--group $\HH= {_\eta}(\bG_{R_n})$ is anisotropic. There exists a family of pro-solvable groups in $k$-vector spaces
$(J_i)_{i=1,..,n}$
such that
 $$
\HH(F_n) \simeq   J_n \rtimes J_{n-1} \rtimes \ldots \rtimes J_{1} \rtimes
( {_z\bG})^{\eta^{geo}}(k) \simeq (J_n \rtimes J_{n-1} \rtimes \ldots
\rtimes J_{1} )  \, \cdot \,  \HH(R_n).$$
\end{theorem}

\begin{proof} Twisting by $z$ we may assume that $z$ is trivial.
It is convenient  to work at a finite level, namely with a cocycle
$
\eta:  \widetilde \Gamma_{n,m} \to {\bG}(\tilde k)
$
as in Remark~\ref{finitelevel}. 

We proceed by induction on $n \geq 0$; the case $n =0$ being obvious. We reason by means of a building argument and we
view $\widetilde{F}_{n,m}$ and its subfield
$F_n= (\widetilde{F}_{n,m})^{\widetilde \Gamma_{n,m}}$
as local complete fields
with the residue fields $\widetilde{F}_{n-1,m}$ and $F_{n-1}$ respectively.
Let ${\cal B}_n=
{\cal B}({\bG}_{\widetilde{F}_{n,m}})$ be the (enlarged) Bruhat-Tits building
of the $\widetilde{F}_{n,m}$--group
${\bG}_{\widetilde{F}_{n,m}}$ \cite[\S 2.1]{T2}.
Recall that ${\cal B}_n$ is equipped with a natural action of 
${\bG}(\widetilde{F}_{n,m}) \rtimes \widetilde{\Gamma}_{n,m}$.
Since $\HH$ is anisotropic the algebraic $F_n$--group
$\HH_{F_n}$
is also anisotropic by \cite[cor. 7.4.3]{GP3}. 
It is shown  in \cite[theo. 7.9]{GP3} 
that the building of $\HH_{F_n}$
inside ${\cal B}_n$ consists of a single point
$\phi$ whose stabilizer is
$\bG\bigl(\widetilde{F}_{n-1,m}[[t_n^{\frac{1}{m}}]] \bigr)$. Since $\HH(F_n)$
stabilizes
$\phi$
it follows that
\begin{equation}\label{id2}
\HH(F_n)= \Bigl\{   g \in \bG\bigl(\widetilde{F}_{n-1,m}
[[t_n^{\frac{1}{m}}]] \bigr)
\, \mid \, \eta(\sigma) \,  \sigma(g) = g
\enskip \forall \sigma \in  \widetilde \Gamma_{n,m}  \Bigr\}.
\end{equation}

We next decompose 
$\bmu_m^{n} =  \bmu_m^{n-1} \times \bmu_m$. 
The second component is a finite $k$-group of multiplicative type
acting on $\bG$ via $\eta^{geo}$. 
We let $\bG_{n-1}$ denote the $k$--subgroup of $\bG$
which is the centralizer of this action \cite[II 1.3.7]{DG}.
The connected component 
of $\bG_{n-1}$ is  reductive
according to \cite{Ri}.
Since the action of $\bmu_m^{n-1}$
on $\bG$ given by  $\eta^{geo}$ commutes with that of $\bmu_m$ 
the $k$--group morphism $\eta^{geo}: \bmu_m^n \to \bG$ factors through
$\bG_{n-1}$.

Denote by $\eta_{n-1}^{geo}$  the restriction of
$\eta^{geo}$ to the $k$--subgroup
$\bmu_m^{n-1}$ 
of $\bmu_m^{n}$.
Set $\tilde \Gamma_{n-1, m} := \bmu_m^{n-1}(\tilde k) \rtimes \Gal(\tilde k/k)$
and consider the loop cocycle
$$\eta_{n-1}: \tilde \Gamma_{n-1,m} \to \bG_{n-1}(\tilde k)$$
attached to $(1, \eta^{geo}_{n-1})$.
We define $$\HH_{n-1,R_{n-1}} = {_{\eta_{n-1}}}(\bG_{n-1,R_{n-1}}).$$

The crucial point for the induction argument is the fact that
the twisted $F_{n-1}$--group  ${\eta_{n-1}\bG_{n-1}}$ 
is anisotropic. 
This is established just as in \cite[theo. 7.9]{GP3}. We look now at the specialization map
$$
sp_n: \HH(F_n) \hookrightarrow \bG\bigl(\widetilde{F}_{n-1,m}
[[t_n^{\frac{1}{m}}]] \bigr)
\to  \bG(\widetilde{F}_{n-1,m}).
$$
Let $P$ be the parahoric subgroup of
 $\HH^{\circ}(F_n)$ attached to the point $\phi$. Since the building of $\HH_{F_n}$
consists of the single point we have $P=\HH^{\circ}(F_n)$.
Recall that the notation $P^*$ stands for the ``pro-unipotent radical''
of $P$ as defined in \S \ref{congruence} of the Appendix.

\begin{claim}
We have $P^* = \ker(sp_n)$ and the
 image of $sp_n$ is  $\HH_{n-1}(F_{n-1})$.
\end{claim}

Because $\dG$ is a $k$--group it is clear that the kernel of the specialization map
$\bG\bigl(\widetilde{F}_{n-1,m}
[[t_n^{\frac{1}{m}}]] \bigr)
\to  \bG(\widetilde{F}_{n-1,m})$ is contained in
$\bG^{\circ}\bigl(\widetilde{F}_{n-1,m}
[[t_n^{\frac{1}{m}}]] \bigr)$.
Since $(\HH/\HH^{\circ})(F_n)$ injects into $(\HH/\HH^{\circ})(\widetilde{F}_{n,m})
=(\bG/\bG^{\circ})(\widetilde{F}_{n,m})$,
the kernel of the specialization map $sp_n$ is contained in
$\HH^{\circ}(F_n)$.
The parahoric subgroup of $\bG^{\circ}(\widetilde F_{n,m})$ attached to
the point $\phi$ is $Q=\bG^{\circ}\Bigl(\widetilde{F}_{n-1,m}
[[t_n^{\frac{1}{m}}]] \Bigr)$
and we have $$
Q^*= \ker\bigl( Q
\to \bG^{\circ}(\widetilde{F}_{n-1,m}) \bigr)
$$
by the very definition of $Q^*$.  Hence $\ker(sp_n)= P \cap Q^*= P^*$
by Corollary
\ref{app-cor} applied to the point  $\phi$.

The group $\HH_{n-1}(F_{n-1})$ is a subgroup
of  $\HH_{n}(F_{n})$ which maps identically to itself by
$sp_n$, so we have to verify
that the  specialization
$h_{n-1}$ of an element
 $h \in \HH(F_n)$ belongs to $\HH_{n-1}(F_{n-1})$.
Specializing (\ref{id2}) at $t_n=0$,
we get

\begin{equation}\label{**}
\eta(\gamma) \, \ {^\gamma}h_{n-1} =
h_{n-1} \quad \forall \gamma \in \widetilde \Gamma_{n,m}.
\end{equation}

\noindent We now apply the relation (\ref{**}) to the generator
$\tau_n$ of the Galois group $\Gal\bigl(\widetilde{F}_{n,m}/
\widetilde{ F}_{n-1,m}((t_n)) \bigr)$; it yields

\begin{equation}\label{***}
\eta(\tau_n)   \, h_{n-1} =  h_{n-1},
\end{equation}
where $\eta(\tau_n)\in {\bG}(\widetilde{k})$, so that
$h_{n-1} \in \bG_{n-1}(\tilde F_{n-1,m})$.
Furthermore, the equality (\ref{**}) restricted to
$\tilde \Gamma_{n-1, m}$ shows that $h_{n-1} \in \HH_{n-1}(F_{n-1} )$.
This establishes the Claim.

We can now finish the induction process.
The group $\HH_{n-1}(F_{n-1})$ is a subgroup of
$\HH(F_n)$, so $$
\HH(F_n) = J_n \rtimes \HH_{n-1}(F_{n-1})
$$
where $J_n:= \ker(sp_n)$ is the ``pro-unipotent radical''  and hence
it is pro-solvable in $k$--spaces.
By using the induction hypothesis, we have
$$
\HH_{n-1}(F_{n-1}) =
(J_{n-1} \rtimes \cdots \rtimes J_1) \rtimes \bG_{n-1}^{\eta_{n-1}^{geo}}(k).
$$
Since $\bG_{n-1}^{\eta_{n-1}^{geo}}= \bG^{\eta^{geo}}$,
we conclude that
$$
\HH(F_n) =  (J_{n} \rtimes \cdots \rtimes J_1) \rtimes \bG^{\eta^{geo}}(k)
$$
as desired.

We have  $\bG^{\eta^{geo}}(k) \subset \HH(R_n)$, so  we get the
second 
equality as well.
\end{proof}

\section{Acyclicity, I}\label{Acyclicity1}

Let $\HH$ be a loop reductive
group scheme. We will denote by $H^1_{toral}(R_n, \HH)$ (resp. $H^1_{toral}(R_n, \HH)_{irr}$) 
the subset of $H^1(R_n, \HH)$ consisting of isomorphism classes of 
$\HH$-torsors $\EE$ such that the twisted $R_n$--group ${_\EE}\HH$ admits a 
maximal torus (resp. admits a maximal torus and is irreducible).

\begin{theorem}\label{acyc1} Let $\HH$ be a loop reductive
group scheme. 
Then the natural map
$$
H^1_{toral}(R_n, \HH)_{irr} \to  H^1(F_n,\HH)
$$
is injective.
\end{theorem}

\begin{proof}
By twisting, it is enough to show that for an irreducible
loop reductive group $\HH$ the canonical map
$H^1_{toral}(R_n, \HH) \to H^1(F_n,\HH)$ has trivial kernel. Indeed reductive  $R_n$--group schemes admitting a maximal torus are precisely
the loop reductive groups \cite[theo. 6.1]{GP3}.
We now reason by successive cases.
\smallskip

\noindent {\it  Case 1}: $\HH$ {\it is adjoint and anisotropic.}
We may view $\HH$  as a twisted form of a Chevalley group scheme
$\dH_{R_n}$
by a loop cocycle $\eta : \pi_1(R_n) \to \Aut(\dH)(\ol k)$.
We have the following commutative diagram of torsion bijections
$$
\begin{CD}
H^1_{toral}\big(R_n, \bAut(\HH)\big) @>>> H^1\big(F_n, \bAut(\HH)\big)  \\
@V{\tau_\eta}V{\simeq}V @V{\tau_\eta}V{\simeq}V \\
H^1_{toral}\big(R_n, \bAut(\dH)\big) @>>> H^1\big(F_n, \bAut(\dH)\big) . \\
\end{CD}
$$
The vertical maps are bijective by \cite[III 2.5.4]{Gir} and Remark \ref{klooptorsion},
while the bottom map is bijective by \cite[theo. 8.1]{GP3}. We thus have a bijection
$$
\psi:H^1_{toral}\big(R_n, {\bAut}(\HH)\big) \simlgr H^1(F_n, {\bAut}(\HH)) .
$$
The exact sequence $1 \to \HH \to {\bAut}(\HH) \to {\bOut}(\HH) \to 1$
gives rise to the commutative diagram of exact sequences of pointed sets
\[
\minCDarrowwidth20pt
\begin{CD}
{\bAut}(\HH)(R_n) @>\delta >>
\bOut(\HH)(R_n) @>{\varphi}>>
H_{\et}^1(R_n, \HH) @>>>
 H_{\et}^1\big(R_n, \bAut(\HH)\big) \\
@VVV \mid \mid && @VVV @VV\psi V \\
{\bAut}(\HH)(F_n) @>\gamma >>
\bOut(\HH)(F_n) @>>>
H^1(F_n, \HH) @>>>
 H^1\big(F_n, \bAut(\HH)\big) .\\
\end{CD}
\]

Let $v \in H_{\et}^1(R_n, \HH)$ be a toral class
mapping to $1 \in H^1(F_n, \HH)$.
Since $\psi$ is bijective there exists $u \in \bOut(\HH)(R_n)$
such that $v= \varphi(u)$ and $u\in {\rm Im}\,\gamma$.
Since
$\bOut(\HH)(R_n)$ is a finite group, the Density Theorem~\ref{density} shows that
${\bAut}(\HH)(R_n)$ and ${\bAut}(\HH)(F_n)$ have the same image in ${\bOut}(\HH)(F_n)$.
So  $u\in {\rm Im}\,\delta$, which implies that
 $\gamma=1 \in H_{\et}^1(R_n,\HH)$.

\smallskip

\noindent{\it Case 2:} $\HH$ {\it is irreducible.}
Set $\ZZ= Z(\HH)$; it is an $R_n$--group of multiplicative type and
we have an exact sequence of $R_n$--group schemes
$$
1 \to \ZZ \buildrel i \over \longrightarrow  \HH \to \HH_{ad} \to 1.
$$
Here the adjoint group $\HH_{ad}$ 
is anisotropic since $\HH$ is irreducible.
This exact sequence gives rise to
the diagram
$$
\begin{CD}
\HH_{ad}(R_n) @>{\varphi_{R_n}}>> H_{\et}^1(R_n, \ZZ) @>{i_*}>> H_{\et}^1(R_n, \HH)
@>>>H_{\et}^1(R_n, \HH_{ad}) \\
@VVV @VV{\simeq}V  @VVV @VVV \\
\HH_{ad}(F_n) @>{\varphi_{F_n}}>> H^1(F_n, \ZZ) @>>> H^1(F_n, \HH)
@>>>H^1(F_n, \HH_{ad}) .\\
\end{CD}
$$
Note that the second vertical map is bijective by
\cite[prop. 3.4.(3)]{GP2} since $\ZZ$ is of finite type
(\cite[ XII, \S 3]{SGA3}).

Let $v \in H_{\et}^1(R_n, \HH)$ be a toral class mapping
to $1 \in H^1(F_n, \HH)$.
Taking into account the adjoint anisotropic case, a diagram chase
provides  an element $u \in
H_{\et}^1(R_n, \ZZ)$ such that $v= i_*(u)$ and $u$
belongs to the image of
the characteristic map $\varphi_{F_n}$. Since
$H_{\et}^1(R_n, \ZZ)$ is an abelian torsion  group, the Density
Theorem ~\ref{density} shows that
$\HH_{ad}(F_n)$ and $\HH_{ad}(R_n)$ have the same image in $H_{\et}^1(R_n, \ZZ)$.
So  $u$ belongs to the image of $\varphi_{R_n},$ and this implies that
$v=i_*(u) =1 \in H_{\et}^1(R_n,\HH)$ as desired.
\end{proof}

\section{Conjugacy of certain parabolic subgroup schemes and maximal split
tori}

\begin{theorem}\label{conj-loop} Let $\HH$ be a loop
 reductive group scheme over $R_n$.
There exists a unique $\HH(R_n)$--conjugacy class of

\smallskip

\noindent
{\rm (a)} Couples $(\LL,\PP)$ where $\PP$ is a minimal parabolic
$R_n$--subgroup scheme
of $\HH$  and $\LL$ is a Levi subgroup of $\PP$
such that $\LL$ is  a loop reductive group scheme.

\smallskip

\noindent
{\rm (b)} Maximal split subtori $\SS$ of $\HH$ such that
  $Z_\HH(\SS)$ is a loop reductive group scheme.
\end{theorem}

\begin{remark} The counter-example in \S\ref{counterexample} shows that the
assumption  that $\LL$ and $Z_{\HH}(\SS)$ be loop reductive
group schemes is not superflous.
\end{remark}

\begin{proof} (i) {\it Reduction 
to the semisimple simply connected case.}
Let $\HH^{sc}$ be the simply connected covering of the derived group scheme
of $\HH,$ and let $\EE$ be the radical torus of $\HH$.
There is a canonical central isogeny \cite[\S 1.2]{H}
$$
1 \to \bmu \to \HH^{sc} \times \EE \buildrel f \over \longrightarrow \HH \to 1 .
$$
Let $(\LL,\PP)$  be a pair where  $\PP$ is  a  parabolic
subgroup of $\HH$ containing a Levi subgroup $\LL$. Then 
$$
f^{-1}(  \PP) = \PP^{sc} \times \EE, f^{-1}(  \LL) = \LL^{sc} \times \EE  
$$
where $\PP^{sc}$ is a minimal parabolic subgroup of the $R_n$--group $\HH^{sc}$ and
$\LL^{sc}$ is a Levi subgroup of $\PP^{sc}.$ 
Conversely, from a couple $(\MM, \QQ)$ for $\HH^{sc}$, 
we can define a couple $\bigl( (\MM \times \EE)/ \bmu , (\QQ \times \EE)/ \bmu )$
for $\HH$. 
By \cite[cor. 6.3]{GP3}, loop group schemes are exactly those carrying a maximal
torus.  Since the last property is insensitive to central extensions \cite[XII.4.7]{SGA3}, 
the correspondence described above exchanges loop objects $\LL$ with loop objects $\LL^{sc}$.
Also it exchanges minimal parabolics of $\HH$ with minimal parabolics of  $\HH^{sc}$.
Thus without loss of generality we may assume that $\HH$ is simply connected.





%

\smallskip

\noindent
(ii) {\it Existence (a).} 
Let $\dH$ 
be the Chevalley $k$--form of $\HH$ and let
$\eta : \pi_1(R_n) \to {\bAut}(\dH)(\overline  k)$
be a loop cocycle such that $\HH  = {_\eta({\dH_{R_n}})}$.
Let $(\bT, \bB)$ be a Killing couple of $\dH$ and  $\Pi \subset
\Delta(\dH,\bT)$ be the base of the 
root system associated to $(\bT, \bB)$.
We denote by $\dH_{ad}$ the adjoint group of $\bH$ and  by   $(\bT_{ad}, \bB_{ad})$ 
the corresponding Killing couple. We have
$\bAut(\bH)= \Aut(\dH_{ad})$.
For each $I \subset \Pi$, we have the standard parabolic
subgroup $\bP_I$ of $\dH$ and its Levi subgroup $\bL_I$, as well as $\bP_{I, ad}$ and $\bL_{I, ad}$  for $\bH_{ad}.$

Let $I\subset \Pi$ be 
the subset of circled vertices in the Witt-Tits diagram of $\HH_{F_n}$.
The version of  the ``Witt-Tits decomposition''  given in
\cite[cor. 8.4]{GP3} applied to $\Aut(\dH_{ad}$)  shows that
$$
[\eta] \in {\rm Im}\Bigl( H^1_{loop}\big(R_n,
{\bAut}(\dH_{ad}, {\bP}_{I,ad}, {\bL}_{I,ad})\big)_{irr} \to
H^1_{loop}\big(R_n, {\bAut}(\dH_{ad})\big) \Bigr).
$$
Thus we may assume that $\eta$ has values in
$${\bAut}(\dH, {\bP}_{I}, {\bL}_{I})(\overline  k)= \break 
{\bAut}(\dH_{ad}, {\bP}_{I,ad}, {\bL}_{I,ad})(\overline  k).$$
The twisted  $R_n$--group schemes $\PP= {_\eta ({\bP}_I)}$ and
$\LL= {_\eta ({\bL}_I)}$ are as desired for
$\PP_{F_n}$ is a minimal $F_n$--parabolic subgroup of
$\HH_{F_n}$ by the definition
of the Witt-Tits index.


\smallskip

\noindent 
(iii) {\it Existence (b).} Consider the pair $(\LL,\PP)$ constructed in (ii) and let 
$\SS$ be the maximal split subtorus of the
radical $\TT$ of $\LL$.
By Proposition~\ref{centralizersplit} we have $Z_{\HH}(\SS)=\LL$ so that 
 $Z_{\HH}(\SS)$ is a loop reductive group. To show that $\SS$
is a maximal split torus of $\HH$ 
it suffices to establish that so is $\SS_{F_n}$. 

Assume that $\SS_{F_n}\subset \SS'$ is a proper inclusion where $\SS'$ is a split torus
in $\HH_{F_n}$. By construction $\PP_{F_n}$ is a minimal parabolic subgroup
over $F_n$. Hence $\LL_{F_n}=C_{\HH_{F_n}}(\SS_{F_n})=C_{\HH_{F_n}}(\SS').$
This implies that $\SS'$ is contained in the radical $\TT_{F_n}$ of $\LL_{F_n}$.
But by Lemma \ref{maxtor}, 
$\SS$ is still maximal split in $\TT$ over $K_n$ and hence over $F_n$ because
$\TT$ is split over a Galois extension $\widetilde{R}_{n,m}/R_n$ for some integer
$m$ -- a contradiction. 

\smallskip

\noindent
(iv) {\it Conjugacy (a).}
 Let $(\LL, \PP)$ be the couple constructed in (ii).
Consider the $R_n$--scheme  $\YY=\HH/\PP$ of parabolic subgroups  of
type ${\rm \bf t}(\PP)$. 
The exact sequence 
$1 \to \PP \to \HH \buildrel f \over \longrightarrow \YY \to 1$
induces  exact sequences of pointed sets 
$$
\begin{CD}
\HH(R_n) @>\psi>> \YY(R_n) @>{\varphi}>> H_{\et}^1(R_n, \PP) @>>> H_{\et}^1(R_n, \HH) \\
&&&& @AA{\simeq}A \\
&&&&H_{\et}^1(R_n, \LL) \\
\end{CD}
$$
(note that the natural mapping $H_{\et}^1(R_n, \LL) \to H_{\et}^1(R_n,\PP)$
is a  bijection  by
\cite[XXVI, 3.2]{SGA3})
and by base change
$$
\begin{CD}
\HH(F_n) @>\psi_{F_n}>> \YY(F_n) @>{\varphi_{F_n}}>> H^1(F_n, \PP)
@>>> H^1(F_n, \HH) \\
&&&& @AA{\simeq}A \\
&&&&H^1(F_n, \LL) \\
\end{CD}
$$

Let $(\MM, \QQ)$ be another couple satisfying the conditions of  Theorem
\ref{conj-loop}. By \cite{GP3},  $\QQ_{F_n}\subset \HH_{F_n}$ is still  
a minimal parabolic subgroup; in particular
$\QQ$ has the same type ${\rm \bf t}(\PP)$ and  hence it corresponds to 
a point  $y \in \YY(R_n)$.

\begin{claim}
$\varphi(y) \in   H^1_{toral}(R_n, \LL) \simeq H^1_{toral}(R_n, \PP) $.
\end{claim}

Indeed, $\varphi(y)$ is  the class of the $\PP$--torsor $\EE:=f^{-1}(y)$.
We can assume without loss of generality that
$\EE$ is obtained from an $\LL$--torsor $\FF$.
Then $\QQ$ is isomorphic\footnote{Surprisingly enough, this compatibility is
 not in Giraud's book. A proof can be found in 
\cite[lemme 4.2.33]{De}.} to
the twist ${_{\FF} \PP}$,
and ${_{\FF} \LL}$ is  a Levi subgroup of the $R_n$--group ${_{\FF} \PP}$.
Since Levi subgroups of  ${_{\FF} \PP}$ are conjugate
under $R_u({_{\FF} \PP})(R_n)$ \cite[XXVI, 1.8]{SGA3}, it follows that
${_{\FF} \LL}$ is $R_n$--isomorphic to $\MM$.
The group scheme ${_{\FF} \LL}$ carries then a maximal
torus and  the claim is proved.

On the other hand, 
since $\PP_{F_n}$ and
$\QQ_{F_n}$ are minimal
parabolic subgroups of $\HH_{F_n}$ 
they are conjugate under $\HH(F_n)$.
Then $y$ viewed as an element of $\YY(F_n)$ is in the image of $\psi_{F_n}$,
hence $\varphi_{F_n}(y)=1$. It follows that
$\varphi(y)$ belongs to the kernel of
$$
H^1_{toral}(R_n, \LL)_{irr} \to H^1(F_n, \LL)
$$
which is trivial by Theorem~ \ref{acyc1}. Thus $y\in {\rm Im}\,\psi$,
i.e. $\PP$ and $\QQ$ are $\HH(R_n)$--conjugate
and so are the couples    $(\LL, \PP)$ and  $(\MM, \QQ)$.


\smallskip

\noindent (v) {\it Conjugacy (b).} We still denote by $(\LL,\PP)$
the couple constructed in (ii).
Let $\SS'$ be a maximal split subtorus of $\HH$
such that  its centralizer $\LL'= Z_\HH(\SS')$ is a loop reductive group
scheme. 
By Lemma \ref{facts}, $Z_{\HH}(\SS')$ is a Levi subgroup of a
parabolic subgroup of $\PP'$ of $\HH$.
By Proposition \ref{flag} (c), $\PP'$ is a minimal parabolic subgroup
of $\HH$.
%
%
By (iv), the couple $(\LL', \PP')$ is conjugate under $\HH(R_n)$ to
  $(\LL, \PP)$. We may thus assume that $\LL=\LL'$, i.e.
$Z_{\HH}(\SS)= Z_{\HH}(\SS')$. It follows $\SS'$ is a  central
split subtorus
of $\LL$, hence  $\SS' \subset \SS$. But $\SS'$ is a maximal split subtorus
of $\HH$, so we conclude that $\SS=\SS'$ as desired.
\end{proof}























\section{Applications to infinite-dimensional Lie Theory}\label{applications}



Throughout this section we assume that $k$ is algebraically closed of characteristic zero, $\dG$ 
is a simple simply connected Chevalley group over $k$, and $\dg$ its Lie algebra. 
We fix integers $n \geq 0$, $m > 0$ and an
$n$--tuple $\bs = (\sigma  _1,\dots,\sigma  _n)$  of commuting
elements of ${\rm Aut}_k(\dg)$ satisfying $\sigma  ^{m}_i = 1.$ Let 
$R=R_n$ and $\widetilde{R}=R_{n,m}$.
Recall that $\widetilde{R}/R$ is Galois and that we can identify $\Gal(\widetilde{R}/R)$ 
with ${(\Z/m\Z )}^n$ via our choice of compatible primitive roots of unity. 

Recall also from the Introduction the multiloop algebra based on $\dg$ corresponding to $\bs$, is
$$
L(\dg,\bs) = \bigoplus_{{(i_1,\dots ,i_n)\in \Z^n}}
\dg_{i_1\dots i_n} \otimes t^{\frac{i_1}{m}}_1 \dots
t^{\frac{i_n}{m}}_n \subset \dg\otimes _k \widetilde{R} 
$$
It is a twisted form of the $R$--Lie algebra $\dg \otimes_k R$ which is split by $\widetilde {R}$.
The $\widetilde{R}/R$ form $L(\dg,\bs) $ is given by a natural loop cocycle 
$$\eta = \eta(\bs) \in Z^1\big(\Gamma, {\bAut}(\dg)(k)\big) \subset 
Z^1\big(\Gamma, {\bAut}(\dg)(\widetilde{R})\big).$$
Since $\bAut(\dg) \simeq {\bAut}(\dG)$ we can also consider by means of $\eta$ the twisted 
$R$--group $\GG = {_\eta}\dG_R.$ As before we denote the Lie algebra of 
$\GG$ by $\gg.$ 
Clearly,
$\gg \simeq L(\dg,\bs)$.

\subsection{Borel-Mostow MAD subalgebras} By a Theorem of Borel and Mostow \cite{BM} there exists a Cartan subalgebra $\dh$ 
of $\dg$ that is  stable under the action of $\bs$ (by which we mean that each of the $\sigma_i$ stabilizes $\dh).$ By restricting $\bs$ to $\dh$ we can consider the loop algebra based on $\dh$ with respect to $\bs,$ 

$$
L(\dh,\bs) = \bigoplus_{{(i_1,\dots ,i_n)\in \Z^n}}
\dh_{i_1\dots i_n} \otimes t^{\frac{i_1}{m}}_1 \dots
t^{\frac{i_n}{m}}_n \subset \dh\otimes _k \widetilde{R} 
$$

Let $\bT$ be the maximal torus of $\dG$ corresponding to $\dh. $ Denote by 
$\bT^{{\bs}}$ (resp.
$\dh^{\bs}$) the fixed point subgroup of $\bT$ (resp. subalgebra of $\dh$ ) 
under $\bs$, i.e the elements of $\bT$ 
(resp. $\dh$) that are fixed by each of the $\sigma_i.$ 
Since the torus  $\bT$ is also $\bs$--stable, just as above,  we can consider its twisted form 
$\TT = {_\eta}\bT_R$ and the corresponding Lie algebra $\hh={_\eta}\dh_R$. 
The same formalism already mentioned yields that $\hh \simeq L(\dh,\bs).$

Let $\TT_d$ be the maximal split torus of $\TT.$ It is easy to see that
$$
\TT_d\simeq\bT^{\bs}_R ={_\eta}(\bT^{\bs}_R)\subset\GG= {_\eta\bG_R}.
$$
 According to Remark \ref{m(S)}  its Lie algebra $\tt_d$ contains a unique maximal subalgebra $\mm$ which is an AD subalgebra of $\gg.$ The description of this algebra is quite simple:
$$ \mm = \dh^{\bs} \otimes_k 1=\dh_{0,\ldots,0} \otimes_k 1 \subset L(\dg , \bs) \simeq \gg.$$ 

By  Theorem \ref{maincorrespondence} $\mm$ is a MAD subalgebra if and only if 
$\TT_d$ is a maximal split torus of $\GG,$ in which case $\mm = \mm(\TT_d).$  We will 
call  MAD subalgebras of a multiloop algebra which are of this form {\it Borel Mostow} 
MAD subalgebras of $\gg.$

Clearly, $Z_\gg(\mm)$ is  precisely the multiloop algebra $L\big(Z_{\dg}(\dh^{\bs}), \bs\big).$
Note that by Proposition~\ref{centralizer}, $Z_{\dg}(\dh^{\bs})$ is 
the Lie algebra of the reductive $k$--group
$\dH := Z_{\dG}(\dh^{\bs})=Z_{\dG}(\bT^{\bs})$ and hence by
twisting we conclude that $Z_{\gg}(\mm)$ is the Lie algebra of 
$Z_{\GG}(\TT_d) =Z_{\GG}(\bT^{\bs}_R)\simeq {_\eta}\dH_R.$

\begin{proposition}\label{propBM} {\rm (1)} 
$Z_{\GG}(\TT_d) $ is a loop reductive group.

\noindent
{\rm (2)} $\mm$ is a MAD subalgebra if and only if the dimension  of $\dh_{0,...,0}$ is maximal among
 the Cartan subalgebras of $\dg$ normalised  by $\bs$.
In particular,  Borel-Mostow MAD subalgebras exist.
\end{proposition}
\begin{proof}
  (1) We have explained  above that $Z_{\GG}(\TT_d) \simeq {_\eta}\dH_R.$ This last group is loop reductive by definition since $\eta$ is a loop cocycle.
 
 \noindent (2) It follows from Theorem~\ref{conj-loop} that 
all maximal split tori in $\GG$ whose centralizers are loop reductive groups and 
corresponding 
MAD subalgebras are conjugate, hence  have the same dimension, say $r$, 
equal to the rank of $\GG$ over $F_n$. 
Since $\mm$ is an AD subalgebra whose centralizer is a multiloop algebra
Theorem~\ref{conj-loop} applied to $Z_{\GG}(\TT_d)$ shows that
$\dim_k \mm=\dim_k(\dh_{0,...,0})
 \leq r$ and hence $\mm$ is a MAD subalgebra if and only if    $\dim_k(\dh_{0,...,0})= r$.
It is then enough to show that there exists a Borel-Mostow AD subalgebra
of rank $r$, that is we need to find  
 a Cartan subalgebra $\dh'$ of $\dg$ normalized by $\bs$ such that
 $\dim_k(\dh'_{0,...,0} )=r$.  
 
 If $r=0$ there is nothing to prove. Assume that $r>0$. Denote by $I$ the type
 of minimal parabolic subgroups of $\GG$ over $F_n$. Fix a Cartan subalgebra
 $\dh_0\subset \dg$, the corresponding maximal torus $\bT_0\subset \bG$ and
 a basis of the root system $\Sigma(\bT_0,\bG)$. 
 In the course of the proof of Theorem~\ref{conj-loop} we showed that  
up to conjugacy by an element of $\bG(k)$, we can assume that 
$\bs$ normalizes 
the standard parabolic group
$\bP_I$ 
and also the standard Levi subgroup $\bL_I$. 

Let $\bS\subset \bT_0$ be the torus consisting of the fixed point subgroup of the radical of $\bL_I$
under $\bs$. Then $\bS_R\hookrightarrow{_{\eta}}(\bL_R)\subset { _{\eta}\bG_R}$ is the 
maximal split torus in the radical of $_{\eta}(\bL_I)_R$.
Since the twist $_{\eta}(\bP_I)_R\otimes_R F_n$ is
a minimal parabolic subgroup of $\GG$ over  $F_n$
and $_{\eta}(\bL)_I\otimes_R F_n$ is its Levi subgroup it follows 
that $\bS_R\otimes_R F_n$ is a maximal split torus of $\GG$ over $F_n$;
in particular  $\dim_k(\bS)=r$.


Let  $\ss\subset \dh_0$ be the Lie algebra of $\bS$. We have
$\dim_k(\ss)=\dim_k(\bS)=r$ and by our construction $\bs$ acts trivially
on $\ss$.
%
 The reductive subalgebra $Z_\dg(\ss)$  is stable under
$\bs$, so the application of Borel-Mostow's theorem provides a 
Cartan subalgebra $\dh'$ of $Z_\dg(\ss)$ stable under $\bs$.
Its fixed subalgebra has dimension $\leq r$ and contains $\ss$,
hence it coincides with $\ss$.
 \end{proof}

According to our  Conjugacy Theorem all Borel-Mostow MAD subalgebras of a multiloop 
algebra  are conjugate under  $\GG(R_n).$ There is a very important class of 
multiloop algebras, the so-called Lie tori, where Borel-Mostow MAD subalgebras 
play a crucial role. 
We now turn our attention to them.\footnote{Lie tori were introduced by Y.~Yoshii   \cite{Y2,Y3} and further studied
by E.~Neher in \cite{ N1, N2}. The terminology is consistent with that of tori in the theory of non-associative algebras, e.g. Jordan tori. But in the presence of algebraic groups, where tori are well defined objects, the terminology is a bit unfortunate.}

\begin{theorem}\label{invariant}
Let $\mathcal{L}$ be a centreless Lie torus which is finitely generated over its centroid. The (relative) type $\Delta$ is an invariant of $\mathcal{L}$. 
\end{theorem}

\begin{proof} After sorting through the several relevant definitions, the Theorem  follows from our conjugacy of Borel-Mostow MAD subalgebras in view of the realization of the Lie tori in question as multiloop algebras as established in \cite{ABFP}.
\end{proof}

The spirit of this result should be interpreted as the analogue that on $\dg$ we cannot choose two different Cartan subalgebras that will lead to root systems of different type.  More generally, it is the analogue of the fact that the relative type of a 
finite-dimensional simple Lie algebra (in characteristic 0) or of a simple algebraic group is an invariant of the algebra or group in question. 

The relevance of centreless Lie tori is that they sit at the ``bottom" of every Extended Affine Lie Algebra (see \cite{AABGP}, \cite{N1} and \cite{N2}). A good example is provided by the affine Kac-Moody Lie algebras. They are of the form (see \cite{Kac})
$$ \mathcal {E} = \cL \oplus kc \oplus kd$$
where $\cL$ is a loop algebra of the form $ L(\dg, \pi)$ for some (unique) $\dg$ and some (unique up to conjugacy) diagram automorphism $\pi$ of $\dg.$ The element $c$ is central and $d $ is a degree derivation for a natural grading of $\cL.$ If $\dh$ is the standard Chevalley split Cartan subalgebra of $\dg$, then $ \mathcal{H} = \dh^\pi + kc + kd$ plays the role of the Cartan subalgebra for $\mathcal{E}.$ 

 \begin{remark}The invariance of the relative type was established in \cite{Als} by using strictly methods from EALA theory. Allison also showed that under the assumption that conjugacy (as established in this paper) holds, any   isotopy between Lie tori necessarily preserves the external root data information. This is a very important result for the theory of EALAs for, together with conjugacy, it yields a very precise description of the group of automorphisms of Lie tori.
 \end{remark}

\section{Acyclicity, II}\label{Acyclicity2}

\begin{theorem}\label{secondmaintheorem}
Let $\HH$ be a loop reductive group scheme over $R_n$. Then
the natural map
$$
H^1_{toral}(R_n, \HH) \to  H^1(F_n,\HH).
$$
is bijective.
\end{theorem}

\begin{remark} {\rm The theorem generalizes (in characteristic $0$) our main result in
 \cite{CGP}.
Indeed, in that paper we showed that if $n=1$ and $\bG$ is a reductive
group over an arbitrary field $k$ of good characteristic then
$H_{\et}^1(R_1,\bG) \to H^1(F_1,\bG)$ is
bijective and that every $\bG$-torsor is toral. The Theorem also generalizes the 
Acyclicity result of \cite{GP3}, which is used in the present proof and covers the case 
when $\HH$ is ``constant".}
\end{remark}

The proof of the theorem is based on the following statement
which generalizes the 
Density Theorem \ref{density} to the case of arbitrary loop reductive
 group schemes, not necessary
anisotropic.

\begin{theorem} \label{acyc2} 
Let $\bH$ be a linear algebraic $k$--group
 whose connect component of the identity is reductive. Let $\eta: \pi_1(R_n) \to \bH(\ol k)$ be 
 a loop cocycle and consider the loop reductive $R_n$--groups   $\HH = {_\eta\bH}_{R_n}$ 
 and $\HH^{\circ}= {_\eta\bH^{\circ}}_{R_n}.$ Let $(\PP, \LL)$ be a couple given by 
 Theorem \ref{conj-loop} for $\HH^{\circ}$.
Then there exists a normal subgroup $J$ of
$\LL(F_n)$ which is a quotient of  a group admitting  a composition serie
whose quotients  are  pro-solvable groups in $k$--vector spaces such that
$$
\HH(F_n)= \Bigl\langle \,  \HH(R_n), \,  J , \, \HH(F_n)^+  \Bigr\rangle
$$
where $\HH(F_n)^+$ stands for the normal subgroup
of $\HH(F_n)$ generated by one parameter additive $F_n$--subgroups.
\end{theorem}

\begin{remark}\label{rem-acyc2}
 {\rm  If $\HH$ is semisimple simply connected, isotropic and $F_n$-simple
we know that $\HH(F_n)/\HH(F_n)^+ \cong \HH(F_n)/R$ \cite[7.2]{G}, where $R$ is
an $R$-equivalence,
  so that
 the group $\HH(F_n)/\HH(F_n)^+$ has finite exponent ({\it ibid}, 7.6).
In this case, the decomposition reads
$\HH(F_n)= \Bigl\langle \,  \HH(R_n),  \, \HH(F_n)^+  \Bigr\rangle$. }
\end{remark}




\begin{proof}
{\it Case $(1)$:}  $\HH$ {\it is a torus} $\TT.$
We leave it  to the
reader to reason by induction
on $n$ to establish the case of a split torus  $\bT=\bG_m^n$
(the case $n=1$ follows from the
identity   $F_ 1^\times= R_1^\times   \cdot 
\ker( k[[t_1]]^\times \to k^\times)$.
Since all finite connected \'etale coverings of $R_n$ are also Laurent polynomial rings over
field extensions of $k$ \cite[lemma 2.8]{GP3}
and the statement is stable 
under products, the theorem also holds
for induced tori.

Let $\TT$ be an arbitrary torus. Since $\TT$  is isotrivial, it is a quotient
of an induced torus $\EE$. We have then an exact sequence
$$1 \to \SS \buildrel i \over \longrightarrow  \EE \buildrel f
\over \longrightarrow  \TT \to 1$$
of multiplicative $R_n$--group schemes.
It gives rise to a commutative diagram
$$
\minCDarrowwidth25pt
\begin{CD}
1 @>>> \SS(R_n) @>{i_{R_n}}>> \EE(R_n)
@>{f_{R_n}}>> \TT(R_n)  @>{\varphi_{R_n}}>> H_{\et}^1(R_n,\SS) @>>> 1  \\
&& @VVV @VVV @VVV @V{\simeq}VV \\
1 @>>> \SS(F_n) @>{i_{F_n}}>> \EE(F_n)
 @>{f_{F_n}}>> \TT(F_n)  @>{\varphi_{F_n}}>> H^1(F_n,\SS) @>>> 1 \\
\end{CD}
$$
with exact rows. Note that the right vertical map is an isomorphism by
\cite[prop. 3.4]{GP2} and that surjectivity on the right horizontal maps is due to the fact
$H_{\et}^1(R_n,\EE)=H^1(F_n,\EE)=1$.
By diagram chasing we see that
$$
\TT(R_n)/ f_{F_n}\bigl( \EE(R_n)\bigr) \simlgr
 \TT(F_n)/ f_{R_n}\bigl( \EE(F_n)  \bigr) .
$$
Therefore the case of the induced torus $\EE$  provides a suitable
group $J$ such that $\TT(F_n)=\TT(R_n)\cdot f_{F_n}(J)$.

{\it Case $(2)$:} $\HH=\LL$ {\it  is irreducible.}
Let $\CC$ be the radical  torus  of $\LL$.
We have an  exact sequence \cite[XXI, 6.2.4]{SGA3}
$$
\begin{CD}
 1 @>>> \bmu @>{i}>> \mathcal{D}\LL \times_{R_n} \CC  @>{f}>> \LL @>>> 1 .\\
\end{CD}
$$
Here $f$ is a natural multiplication map and $\bmu$ is its kernel.
It gives rise to a commutative diagram of  exact sequences of pointed sets
$$
\minCDarrowwidth20pt
\begin{CD}
 (\mathcal{D}\LL \times \CC)(R_n)
 @>{f_{R_n}}>> \LL(R_n)  @>{\varphi_{R_n}}>> H_{\et}^1(R_n,\bmu) @>{i_{*,R_n}}>>
 H^1_{loop}(R_n,\mathcal{D}\LL\times \CC) \\
 @VVV @VVV @V{\simeq}VV @VVV \\
  (\mathcal{D}\LL \times \CC)(F_n)
 @>{f_{F_n}}>> \LL(F_n)  @>{\varphi_{F_n}}>> H^1(F_n,\bmu) @>{i_{*,F_n}}>>
 H^1(F_n,\mathcal{D}\LL \times \CC) . \\
\end{CD}
$$
Note  that the image of the map  $H_{\et}^1(R_n,\bmu) \to H_{\et}^1(R_n,\mathcal{D} \LL)$ is
contained in
$H^1_{toral}(R_n,\mathcal{D}\LL)$. 
So taking into consideration  Theorem \ref{acyc1} (applied to the irreducible
loop reductive group scheme $\mathcal{D}\LL$ and chasing the above diagram we see that
$$
\LL(R_n)/ f_{R_n}\bigl( (\mathcal{D}\LL)(R_n) \times \CC(R_n) \bigr) \simlgr
 \LL(F_n)/ f_{F_n}\bigl( (\mathcal{D}\LL)(F_n) \times \CC(F_n) \bigr).
$$
The case of $\mathcal{D}\LL$ done in Proposition \ref{density}
together with the case of the torus $\CC$ provide a suitable normal
group $J$ such that $\LL(F_n)= \LL(R_n)\cdot J$.

{\it Case $(3)$}. $\HH=\HH^{\circ}.$ Since $\HH$ is loop reductive by assumption
it suffices to  observe  that
$\HH(F_n)$ is
generated by $\LL(F_n)$ and  $\HH^+(F_n)$ \cite[6.11]{BT73}.

{\it Case $(4)$}. For the general case it remains to show
that for an arbitrary element $g\in {\HH}(F_n)$ the coset
$g{\HH}^{\circ}(F_n)$ contains at least one $R_n$--point of $\HH$.

Let $\SS$ be the maximal split torus of the radical of $\LL$.
The torus $g\SS_{F_n}g^{-1}\subset {\HH}^{\circ}_{F_n}$ is maximal split,
hence 
$g\SS_{F_n}g^{-1}=
g_1\SS_{F_n}g^{-1}_1$ for some $g_1\in \HH^{\circ}(F_n)$. Thus replacing 
$g$ by $g_1^{-1}g$ if necessary, we may
assume that $g{\SS}_{F_n}g^{-1}={\SS}_{F_n}$. Then we also have
$g(\LL_{F_n})g^{-1}=\LL_{F_n},$ so that  $g\in N_{\HH}(\LL)(F_n)$.

The torus $\SS$ is clearly normal in $N_{\HH}(\LL)$.
Hence we have an exact sequence
$$
1\longrightarrow \SS \longrightarrow N_{\HH}(\LL)
\longrightarrow \HH' :=N_{\HH}(\LL)/\SS\longrightarrow 1.
$$
Note that since $H_{\et}^1(R_n,\SS)=1$, the natural maps $N_{\HH}(\LL)(R_n)\to
\HH'(R_n)$ and $N_{\HH}(\LL)(F_n)\to
\HH'(F_n)$ are surjective.
Furthermore, $\HH'$ satisfies all conditions of Theorem~\ref{density},
so that the required fact follows immediately from that theorem applied
to $\HH'$ and from the surjectivity of the above maps.
\end{proof}

We can proceed to the proof of Theorem \ref{secondmaintheorem}.

\begin{proof}
\smallskip

\noindent{\it Injectivity:}
By twisting, it is enough to show that the natural map
$H^1_{toral}(R_n, \HH) \to H^1(F_n,\HH)$ has trivial kernel.

\smallskip \noindent
We first assume that  $\HH$ is adjoint.
We may view  $\HH$  as the twisted form of a Chevalley group scheme
$\dH_{R_n}$
by a loop cocycle $\eta : \pi_1(R_n) \to {\rm Aut}\big(\dH(\ol k)\big)$.
The same reasoning given in Case 1 of the proof of Theorem \ref{acyc1} shows  that we have a natural bijection
\begin{equation}\label{bij}
H^1_{toral}\big(R_n, {\bAut}(\HH)\big) \simlgr H^1\big(F_n, {\bAut}(\HH)\big) .
\end{equation}
The exact sequence
$$
1 \to \HH \to {\bAut}(\HH) \to {\bOut}(\HH) \to 1$$
gives rise to a commutative diagram of exact sequence of pointed sets
$$
\minCDarrowwidth20pt
\begin{CD}
{\bAut}(\HH)(R_n) @>\gamma>>
{\bOut}(\HH)(R_n) @>{\varphi}>> H_{\et}^1(R_n, \HH) @>>>H_{\et}^1(R_n,
{\bAut}(\HH)) \\
@VVV \mid \mid && @VVV @VVV \\
{\bAut}(\HH)(F_n) @>\psi>> {\bOut}(\HH)(F_n) @>>> H^1(F_n, \HH)
@>>>H^1(F_n, {\bAut}(\HH)) .\\
\end{CD}
$$
Let  $v \in H_{\et}^1(R_n, \HH)$ be a toral class mapping to $1 \in H^1(F_n, \HH)$.
In view of bijection~(\ref{bij}) there exists $u \in {\rm Out}(\HH)(R_n)$
such that $v= \varphi(u)$ and $u$ belongs to the image of $\psi$.
 Since
${\rm Out}(\HH)(R_n)$ is a finite group, the Density Theorem~\ref{acyc2}
shows that
${\rm Aut}(\HH)(R_n)$ and ${\rm Aut}(\HH)(F_n)$ have same image
in ${\rm Out}(\HH)(F_n)$.
So  $u$ belongs to the image of $\gamma$,
hence $v=1 \in H_{\et}^1(R_n,\HH)$.

\smallskip

\noindent Let now $\HH$ be an arbitrary reductive group.
Set $\CC= Z(\HH)$. This is an $R_n$--group of multiplicative type and
we have an exact (central) sequence of $R_n$-group schemes
$$
1 \to \CC \buildrel i \over \longrightarrow  \HH \to \HH_{ad} \to 1.
$$
This exact sequence gives rise to the diagram of exact sequences of pointed sets
$$
\begin{CD}
\HH_{ad}(R_n) @>{\varphi_{R_n}}>> H_{\et}^1(R_n, \CC) @>{i_*}>> H_{\et}^1(R_n, \HH)
@>>>H_{\et}^1(R_n, \HH_{ad}) @>{\Delta}>> H_{\et}^2(R_n, \CC) \\
@VVV @VV{\simeq}V  @VVV @VVV  @VV{\simeq}V \\
\HH_{ad}(F_n) @>{\varphi_{F_n}}>> H^1(F_n, \CC) @>>> H^1(F_n, \HH)
@>>>H^1(F_n, \HH_{ad}) @>{\Delta_{F_n}}>> H_{\et}^2(F_n, \CC).\\
\end{CD}
$$
The isomorphisms $H_{\et}^i(R_n, \CC) \cong H^i(F_n, \CC)$ comes from
\cite[prop. 3.4.(3)]{GP2} for $i=1,2$.

Let  $v \in H_{\et}^1(R_n, \HH)$ be a toral class
mapping to $1 \in H^1(F_n, \HH)$.
Taking into account the adjoint case, a diagram chase provides  $u \in
H_{\et}^1(R_n, \CC)$ such that $v= i_*(u)$ and $u_{F_n}$
belongs to the image of
the characteristic map $\varphi_{F_n}$.
 Since
$H_{\et}^1(R_n, \CC)$ is an abelian torsion  group,
the Density Theorem~\ref{acyc2} shows that
$\HH_{ad}(F_n)$ and $\HH_{ad}(R_n)$ have the same images in $H_{\et}^1(R_n, \CC)$.
So  $u$ belongs to the image of $\varphi_{R_n}$. Hence
$v=i_*(u) =1 \in H_{\et}^1(R_n,\HH)$.

\noindent{\it Surjectivity:} Follows by a simple chasing in the diagrams
above.
\end{proof}

\noindent{\it Question.} Assume that $\HH$ is loop semisimple
simply connected, isotropic and $F_n$-simple.
Let $\HH(R_n)^+ \subset \HH(R_n)$ be the (normal) subgroup generated by
the $R_u(\PP)(R_n)$ where $\PP$ runs over the set of parabolic subgroups
of $\HH$ considered in Theorem \ref{conj-loop}.
Is the map
$$
 \HH(R_n) / \HH(R_n)^+  \to \HH(F_n) / \HH(F_n)^+
$$
an isomorphism?

Note that the map  is surjective by Remark \ref{rem-acyc2}. 



\section{Appendix: Greenberg functors, Bruhat-Tits theory and pro-unipotent
radicals}


We are given a complete discrete valuation field $K$ of valuation ring
$O=O_K$ and of perfect residue field $k=O/\pi O$. Here $\pi\in O$
is a uniformizer.
In the inequal characteristic case denote by $e_0$ the absolute ramification index of $O$, i.e.
 $p= u \pi^{e_0}$ for a unit $u \in O$ where $p= {\rm char}(k)$;
in the equal characteristic case, put $e_0=1$.
We denote by $O^{sh}$ the strict  henselization  of $O$, or in other words, its
maximal unramified extension.

\subsection{Greenberg functor}

We recall here basic facts, see the references
\cite{Gb},  \cite[\S III.4]{M2}, \cite{BLR},  \cite{B}.

Assume first that we are in  the  unequal characteristic case,
that is $K$ is of characteristic $0$ and $k$ is of characteristic $p>0$.

For each $k$-algebra $\Lambda$  and $r \geq 0$, we denote by $W_r(\Lambda)$
 the group of Witt vectors of length $r$ and by $W(\Lambda)= \limproj W_r(\Lambda)$
the ring of Witt vectors (see \cite[\S II.6]{Se2}).
There exists a unique ring homomorphism  $W(k) \to O$ commuting with the projection
on $k=W_0(k)$ ({\it ibid}, II.5).

Let $\SS$ be an affine  $W(k)$-scheme.
Recall that for each  $r \geq 0$,
 the functor $\kalg \to Sets$ given by
$$\Lambda \to  \SS(W_r(\Lambda))$$ is
representable by an affine  $k$--scheme
${\rm Green}_r(\SS)$.
The projective limit
$${\rm Green}\,(\SS):\,= \underset{r}{\limproj} {\rm Green}_r(\SS)$$ is a scheme
which satisfies  ${\rm Green}(\SS)(\Lambda)= \SS(W(\Lambda))$.
If $\XX$ is an affine  $O$--scheme, we deal also with the relative versions
of the Greenberg functor
$$
\uG_r(\XX):=  {\rm Green}_r( \prod\limits_{O/W(k)} \XX) , \quad
\uG(X):=  {\rm Green}\,( \prod\limits_{O/W(k)} \XX).
$$
We have $\uG_r(\XX)(k)= \XX(O/p^r O)$ and $\uG(\XX)(k)= \XX(O)$.
We also have $\uG(\Spec(O))= \Spec(k)$; if $\XX$ is a $O$--group scheme, then 
  $\uG(\XX)$ and the  $\uG_r(\XX)$ carry a natural $k$-group structure \cite[4.1]{B}.

\begin{lemma}\label{Greenberg2} Let $L/K$ be a finite extension,
$O_L$ the valuation ring of $L$ and $l/k$ the corresponding
residue extension.
Let $\YY/O_L$ be an affine scheme. Let $\uH/l$ be
the relative Greenberg functor
of $\YY$ with respect to $W(l)$.
Then  we have  natural isomorphisms of $k$-schemes (for all  $r\geq 1$)
$$
\uG_r\Bigl( \prod\limits_{O_L/O} \YY \Bigr) \simeq  \prod\limits_{l/k}
\uH_r(\YY) ,
\quad
\uG\Bigl(\prod\limits_{O_L/O} \YY \Bigr) \simeq  \prod\limits_{l/k} \uH(\YY) .
$$
In particular if $k=l$ then we have $\uG_r\Bigl( \prod\limits_{O_L/O} \YY \Bigr) =
\uH_r(\YY)$ and
$\uG\Bigl(\prod\limits_{O_L/O} \YY \Bigr) \simeq   \uH(\YY)$.
\end{lemma}

\begin{proof}
We have a commutative square
$$
\begin{CD}
 O @>>> O_L \\
@AAA @AAA \\
W(k) @>>> W(l). \\
\end{CD}
$$
So by the functorial properties of the Weil restriction, we have
\begin{equation}\label{weil}
\prod\limits_{O/W(k)} \prod\limits_{O_L/O} \YY
=\prod\limits_{O_L/W(k)}  \YY
= \prod\limits_{W(l)/W(k)} \prod\limits_{O_L/W(l)} \YY.
\end{equation}
Let $\Lambda$ be a $k$--algebra. Using (\ref{weil}) and
the definitions of the Greenberg functors, we have
\begin{eqnarray} \nonumber
 \uG_r\Bigl( \prod\limits_{O_L/O} \YY \Bigr)(\Lambda)
&=& {\rm Green}_r\Bigl( \prod\limits_{O_L/W(k)} \YY \Bigr)(\Lambda) \\ \nonumber
&=& \Bigl( \prod\limits_{W(l)/W(k)} \prod
\limits_{O_L/W(l)} \YY \Bigr)(W_r(\Lambda) )\\ \nonumber
&=& \Bigl( \prod\limits_{O_L/W(l)} \YY \Bigr)\bigl( W(l) \otimes_{W(k)}
W_r(\Lambda) \bigr). \\ \nonumber
\end{eqnarray}

\noindent Since $W_r(\Lambda)$ is a $W_r(k)$-module, we have
$$W(l) \otimes_{W(k)}   W_r(\Lambda)
= \break W_r(l) \otimes_{W_r(k)}   W_r(\Lambda) = W_r( \Lambda \otimes_k l)$$
by \cite[1.5.7]{I}.
Hence
$$
\uG_r\Bigl( \prod\limits_{O_L/O} \YY \Bigr)(\Lambda)=
 \Bigl( \prod\limits_{O_L/W(l)} \YY \Bigr)(  W_r(\Lambda \otimes_k l))
= R_{l/k}\bigl( \uH_r\bigr) (  \Lambda)
$$
as desired. By passing to the limit, we get the  second identity.
\end{proof}

\begin{lemma}\label{Greenberg1}
{\rm (1)} Let $\XX/O$ be an affine  scheme of finite type
such that   $\XX_K=\emptyset$. Then $\uG(\XX)=\emptyset$.

\smallskip

{\rm (2)} Let $\NN/O$ be an affine  group scheme of finite type
such that   $\NN_K=\Spec\,(K)$. Then $\uG(\NN)=\uG(\Spec\,(O))= \Spec(k)$.
\end{lemma}

\begin{proof}
(1) We have $\XX=\Spec\,(A)$ where
$A$ is an $O/\pi^dO$--algebra of finite type for $d$ large enough
Put $r_0= d \,  e_{0}$.
Then  $p^{r_0} A=0$. For a $k$--algebra $\Lambda$
 we have by definition
$$
\uG(\XX)(\Lambda)= \Hom_{O}(A, W(\Lambda)\otimes_{W(k)} O).
$$
But $W(\Lambda)\otimes_{W(k)} O$ is $p$--torsion free, so
$\uG(\XX)=\emptyset$.

\smallskip

\noindent (2) We have $N=\Spec(B)$ and we have the decomposition
$B= O \oplus I$ where $I$ is the kernel of the co-unit of 
the corresponding Hopf algebra. The $O$-module  $I$ is an ideal of $B$  which is an
$O/\pi^dO$-algebra of finite type.
The same reasoning as above shows that $$
\begin{array}{ccl}
\uG(\NN)(\Lambda) &= &\Hom_{O}(B, W(\Lambda)\otimes_{W(k)} O)\\
& = &
\Hom_{O}(O, W(\Lambda)\otimes_{W(k)} O)\\
&=& \uG(\Spec(O))(\Lambda).
\end{array}
$$
Thus  $\uG(\NN)= \uG(\Spec(O))$ which is nothing but    $\Spec(k)$ as explained above.
\end{proof}

Secondly, assume that $k$ and $K$ have the same
characteristic ($0$ or $p>0$) and we still assume that $k$ is perfect.
Then  $k$ embeds in $O$ (in a 
unique way, \cite[21.5.3]{EGAIV}) and for an $O$-scheme $\XX$ the functors
$$\uG(\XX):= \prod\limits_{O \mid k} \XX \mbox{\ \ and\ \ }
\uG_r(\XX):= \prod\limits_{O /\pi^{r} O \,\,
\mid  \, k} (\XX \times_O O /\pi^{r} O)$$
play the desired role \cite[\S 9.6]{BLR} and allow us to write
$$
\XX(O)= \underset{r}{\limproj}\,\,  \XX(O/\pi^r O) \, = \, \underset{r}{\limproj}\,\, \uG_r(\XX)(k)
$$
where the $\uG_r(\XX)$ are $k$-schemes (by Weil restriction considerations \cite[\S 7.6]{BLR}).
The two lemmas are true as well.

\subsection{Congruence filtration}\label{congruence}
Let $\dG$ be a reductive $K$--group and denote by ${\cal B}={\cal B}(\dG,K)$
its (extended) Bruhat-Tits building.
Let $x$ be a point  of  ${\cal B}$ and denote by $P_x$ the parahoric subgroup
$$
P_x=
\Bigl\{ \, g \in \dG(K) \, \mid \,  g(x)= x \, \Bigr\}.
$$
Denote
by  ${\mathfrak P}_x$ the canonical smooth group scheme
over $O$ defined by Bruhat-Tits
\cite[\S 5.1]{BT2} with generic fiber $\dG$ and such that
${\mathfrak P}_x(O)=P_x$ or,
more precisely,
$$
 {\mathfrak P}_x(O^{sh})=
\Bigl\{ \, g \in \dG(K^{sh}) \, \mid \,  g(x)= x \, \Bigr\}
$$
where $x$ is viewed as an element in ${\cal B}(\dG,K^{sh})$ via
the canonical mapping ${\cal B}(\dG,K)\hookrightarrow {\cal B}(\dG,K^{sh})$.
Since ${\mathfrak P}_x$ is smooth
we  have
$$
{\mathfrak P}_x(O)= \limproj\limits_{n \geq 1} {\mathfrak P}_x(O/\pi^{n}O)
$$
and the transition maps ${\mathfrak P}_x(O/\pi^{n+1}O) \to
{\mathfrak P}_x(O/\pi^{n}O)$ are
surjective with kernel ${\rm Lie}({\mathfrak P}_x) \otimes_O k$
(\cite[III.4.3]{M2}) 

The  application of the relative Greenberg functor to
the smooth affine group scheme ${\mathfrak P}_x$
defines a projective system of affine $k$-groups $\bP_{x,n}$ ($n \geq 1$)
such that
$$
\bP_{x, n}(k)={\mathfrak P}_x( O/\pi^{n e_0} O).
$$
The $\bP_{x,n}$ are smooth according to \cite[Lemme 4.1.1]{B}.
The kernel $\bP_{x,n+1/n}$ of the transition maps
$\bP_{x,n+1} \to \bP_{x,n}$ are
$k$-unipotent abelian groups which are successive extensions of the
vector group of ${\rm Lie}({\mathfrak P}_{x}) \otimes_O k$
({\it ibid.} or \cite[III.4.3]{M2}).

For each $n \geq 1$, we denote by
 $\bR_{n,x}:=R_{u}(\bP_{x,n})$ the unipotent radical
of ${\bP}_{x,n}$; since $k$ is perfect, it is defined over $k$ and split 
\cite[IV.2.3.9]{DG}.
The quotient $\bM_x$  of $\bP_{x,n}$ by $\bR_{x,n}$ is independent of $n$.
It is  nothing but  the quotient of the
 special fiber of ${\mathfrak P}_x$ by its $k$-unipotent radical $R_x$.
The $k$-group $\bM_x^{\circ}$ is reductive according to \cite[4.6.12]{BT2}.

We consider the ``maximal  pro-unipotent normal subgroup''
$$ 
P_x^{*}:= \ker\Bigl( {\mathfrak P}_x(O) \to \bM_x(k)\Bigr)
$$
which is of analytic nature.
Denote by
$$
\bP_x/k:= \underset{n\geq 1}{\limproj} \bP_{x,n}
$$ and by
$\bP_x^*/k= \ker\bigl(\bP_x \to \bM_x)$. 
By construction we have  $P_x^{*}= \bP_x^{*}(k)$.

\begin{lemma} For each $n \geq 1$, there is a short exact sequence
of affine $k$--groups $$
1 \to \ker(\bP_x \to \bP_{x,n})  \to  \bP_x^{*} \to \bR_{x,n} \to 1 .
$$
\end{lemma}

\begin{proof} Apply the snake lemma to the commutative diagram of $k$--groups
$$
\begin{CD}
 1 @>>> \bP_x^* @>>> \bP_x @>>> \bM_x \to 1 \\
&& @VVV @VVV \mid \mid \qquad \\
 1 @>>> \bR_{x,n} @>>> \bP_{x,n} @>>> \bM_x \to 1.  \\
\end{CD}
$$
\end{proof}

\begin{lemma}\label{prosolv}
The $k$-group  $\bP_x^*$ is the unique maximal split pro-unipotent
closed normal subgroup of
the pro-algebraic affine $k$-group $\bP_x$.
 \end{lemma}

\begin{proof}
Since
$$
\ker(\bP_x \to \bP_{x,1}) = \underset{n}{\limproj} \ker(\bP_{x,n} \to
\bP_{x,1})$$
is pro-unipotent, the above exact sequence
shows that $\bP_x^{*}$ is pro-unipotent. Let $\bU_x$ be a pro-unipotent
normal closed subgroup of $\bP_x$.  The image of
$\bU_x$ by the map
$\bP_x \to \bM_x$ is a normal unipotent connected $k$-subgroup.
 Since $\bM_x^{\circ}$ is reductive, its image is trivial.
Therefore $\bU_x \subset \bP^*_x$ which completes the proof.
\end{proof}

\subsection{Behaviour under a Galois  extension}

Just as does the whole theory, the construction of $P^*_x$
 has a very nice behaviour  with respect to
unramified extensions of $K$.
The behaviour under a given  tamely ramified
 finite Galois field extension
$L/K$  is subtle. Since such an extension is a tower of an
unramified extension and a totally ramified one,
we may concentrate on the case when $L/K$ is totally (tamely) ramified.
Then $L/K$ is cyclic of degree $e$ invertible in
$k=\overline{K}=\overline{L}$.
The Galois group $\Gamma= \Gal(L/K)$ acts on the
building ${\cal B}(\dG,L)$.
The Bruhat-Tits-Rousseau theorem (\cite[\S 5]{Ro}, see also \cite{Pr})
states that the natural map
$$
j: {\cal B}(\dG,K) \to {\cal B}(\dG,L)
$$
induces a bijection ${\cal B}(\dG,K) \simlgr {\cal B}(\dG,L)^\Gamma$.
For $z\in  {\cal B}(\dG,L)$, we denote by $Q_z$ the parahoric subgroup
of $\dG(L)$ and  by ${\mathfrak Q}_z$ the canonical group scheme
over $O_L$
attached to the
point  $z$.

For $\sigma \in \Gamma$, we have $\sigma( Q_z) = Q_{\sigma(z)}$.
Hence for the canonical group schemes over $O_L$ attached to $z$ and
$\sigma(z)$, there is a   natural cartesian square
$$
\begin{CD}
 {\mathfrak Q}_{\sigma(z)} @>{f_{\sigma, z}} >> {\mathfrak Q}_z \\
@VVV @VVV \\
\Spec(O_L) @>{ (\sigma^{-1})^*}>>  \Spec(O_L) .\\
\end{CD}
$$
Put $y=j(x) \in {\cal B}(\dG,L)^\Gamma$.  We then have an $O$-action of $\Gamma$ on the scheme ${\mathfrak Q}_y$. We note that

\begin{eqnarray}\label{fact434}
P_x=  \dG(K) \cap Q_y &=& \dG(L)^\Gamma \cap Q_y = Q_y^\Gamma.
\end{eqnarray}

\noindent As above we consider the groups $\bQ_{y,n}$ and their projective limit
$\bQ_y$. Since $k$ is the residue field of $O_L$, all $\bQ_{y,n}$
and $\bQ_y$ are $k$-groups. The action of $\Gamma$ on ${\mathfrak Q}_y$
induces its action on $\bQ_{y,n}$, hence on $\bM_y$ where $\bM_y$ stands for
the reductive $k$--group attached to $y$, and on their projective limit
$\bQ_y$.
By Lemma \ref{prosolv},
$\bQ_y^{*}$ is a characteristic $k$-subgroup of  $\bQ_y$, hence $\Gamma$
also acts  on the pro-algebraic $k$-group $\bQ^*_y$.
Our goal is to prove the following fact:

\begin{proposition}\label{compatibility}
There is a natural closed embedding $\bP_x \to \bQ_y$ and
we have $$
\bP_x^* = \bP_x \cap \bQ_y^{*}.
$$
This gives rise to an isomorphism $\bM_x \simlgr \bM_y^\Gamma$.

\end{proposition}

By taking $k$--points we get the following wished compatibility, namely.

\begin{corollary}\label{app-cor} We have
$$
P_x^{*} \simlgr P_x \cap Q_y^{*}.
$$
\end{corollary}

Consider  the Weil restriction ${\mathfrak J}_x:=
\Pi_{O_L/O}\bigl( {\mathfrak Q}_y \bigr)$ and recall  it  is
a smooth $O$-scheme \cite[\S 2.5]{Yu}.
Let $\mathfrak N$ be the kernel of
the natural map ${\mathfrak P}_x \to {\mathfrak J}_x$, its generic fiber
is trivial. As above, applying the Greenberg functors
to the $O$-schemes ${\mathfrak J}_x$ and ${\mathfrak N}$ we get $k$-groups
$\bJ_{x,n},\ \bJ_x$ and $\bN_n,\ \bN$.

Since the Greenberg functor is left exact, we get
an exact sequence
$$
1 \to \bN \to  \bP_{x} \to \bJ_x.
$$
Since ${\mathfrak N}_K=1$,  we have $\bN=1$
according to Lemma   \ref{Greenberg1} (2).
Hence we may view
$\bP_x$ as a closed subgroup of
$\bJ_x$. 
But according to Lemma \ref{Greenberg2},  $\bJ_{x,n}$ is
nothing but $\bQ_{y,n}$. This implies $\bJ_x$ is isomorphic
in a natural way to $\bQ_y$. Thus we have constructed a natural closed
embedding $\bP_x\to \bQ_y$.

Define the $k$-subgroups $\bQ_y^\Gamma:= \underset{n}{\limproj}\,
\bQ_{y,n}^\Gamma$
and $(\bQ_y^*)^\Gamma= { \bQ_y^\Gamma \cap \bQ_y^*}$ of $\bQ_y$ and $\bQ^*_y$
respectively.

\begin{lemma} \label{suite}
{\rm (1)}  If $k'/k$ is a finite extension of fields, the projective system
$\bigl( \bQ_{y,n}^\Gamma(k') \bigr)_{n \geq 1}$
has surjective transitions maps. 
Therefore the projective system
of $k$-groups
$\bigl( \bQ_{y,n}^\Gamma \bigr)_{n \geq 1}$
has surjective transitions maps.

\smallskip

\noindent
{\rm (2)} If $k'/k$ is a field finite extension,  we have an exact sequence
$$
1 \to (\bQ_y^*)^\Gamma(k')  \to \bQ_y^\Gamma(k')
 \to  \bM_y^\Gamma(k')\to 1;
$$ 
hence the sequence of the pro-algebraic $k$-groups
$$
1 \to (\bQ_y^*)^\Gamma  \to\bQ_y^\Gamma \to
 \bM_y^\Gamma \to 1
$$ 
is also exact.

\smallskip

\noindent
{\rm (3)} The algebraic $k$--group $\bM_y^\Gamma$ is smooth and 
 its connected component of the identity is reductive.

\end{lemma}

\begin{proof}

\smallskip

\noindent(1) Since Bruhat-Tits theory is insensitive to finite unramified extensions, 
 we may assume without loss of generality that $k=k'$. Since
 $\bQ_{y, n+1/n}$ is a $k$-split unipotent group, we have an
 exact sequence
$$
1 \to \bQ_{y, n+1/n}(k)  \to  \bQ_{y,n+1}(k) \to \bQ_{y,n}(k) \to 1.
$$
It gives rise to the exact sequence of pointed sets
$$
1 \to  \bQ_{y,n+1/n}(k)^\Gamma  \to  \bQ_{y,n+1}(k)^\Gamma \to
\bQ_{y,n}(k)^\Gamma
 \to H^1\big( \Gamma,  \bQ_{y,n+1/n}(k)\big).
$$
Since $\bQ_{y,n+1/n}(k)$ admits a characteristic central composition serie
in $k$-vector spaces  and the order
of $\Gamma$ is invertible in $k$, the right hand side is trivial.
A fortiori, the system $(\bQ_{y,n}^{\Gamma})$ of $k$--groups is surjective
(because $\bQ_{y,n}^{\Gamma}(k)=\bQ_{y,n}(k)^{\Gamma}$).

\smallskip

\noindent(2) By part (1), the map $ \bQ_y^\Gamma(k) \to
(\bQ_{y,1})^\Gamma(k)$ is surjective.
The same argument as in (1) shows that $ (\bQ_{y, 1})^\Gamma (k)\to
 \bM_y^\Gamma(k)$ is also surjective. By taking the composition of these maps we conclude
the map $ \bQ_y^\Gamma(k) \to  \bM_y^\Gamma(k)$ is surjective
 whence the desired exactness of both sequences.

\smallskip

\noindent (3)
The group  $\Gamma$ may be viewed as a finite abelian constant group scheme
whose order  is invertible in $k$. Hence $\Gamma$
is also a (smooth) $k$-group of multiplicative type.
Since $\bM_y$ is affine and smooth,
 Grothendieck's theorem of smoothness of centralizers
\cite[XI, 5.3]{SGA3} shows that $\bM_y^\Gamma$ is smooth.
Its connected component of the identity is reductive by
a result of Richardson \cite[prop. 10.1.5]{Ri}.
\end{proof}

We can now proceed to the proof of
Proposition \ref{compatibility}.

\begin{proof}
We have  to show that our  closed embedding $\bP_x \to \bQ_y$ which we
constructed above induces an isomorphism  $
\bP_x^* \simlgr \bP_x \cap \bQ_y^{*}.$ Since $\bP_x\cap \bQ_y^*$
is a normal closed split pro-unipotent subgroup of $\bP_x$
it is contained in $\bP^*_x$. Hence it remains only
to show that $\bP^*_x\subset \bQ^*_y$.


We now recall from (\ref{fact434}) that $P_x=Q_y^\Gamma$ and $\bQ_y^\Gamma(k)=\bQ_y(k)^{\Gamma}=Q_y^{\Gamma}$.
By Lemma~\ref{suite},  $\bQ_y^\Gamma(k)$ 
projects onto $\bM_y^\Gamma(k)$,
 so the composite map 
$$
P_x=\bP_x(k) \to \bQ_y^\Gamma(k) \to  \bM_y^\Gamma(k)
$$ is surjective.
Since this is true  for all finite extensions of $k$,
the homomorphism of $k$-algebraic groups  $\bP_x \to \bM_y^\Gamma$ is surjective.
But  $({\bM_y^\Gamma})^{\circ}$ is reductive, hence this map is trivial on
the pro-unipotent radical $\bP_x^*$. We get then a surjective map
 $\bM_x \to  \bM_y^\Gamma$ 
and also a homomorphism $\bP_x^* \to (\bQ_y^*)^\Gamma\subset \bQ^*_y$
as required. 
\end{proof}

\bigskip

\bigskip





\begin{thebibliography}{ABCD}


\bibitem[Als]{Als}
B.~Allison, {\it Some isomorphism invariants of Lie tori},
Jour. Lie Theory (to appear).

\bibitem[AABGP]{AABGP}
B. Allison, S.~Azam, S.~Berman, Y.~Gao and A.~Pianzola,
\emph{Extended affine Lie algebras and their root systems},
Mem.~Amer.~Math.~Soc. {\bf 126} \#603 (1997).

\bibitem[ABFP]{ABFP}
B.~Allison, S.~Berman, J.~Faulkner, A.~Pianzola,
\emph{Multiloop realization of extended affine Lie algebras and Lie tori},
Trans. Amer. Math. Soc.  \textbf{361} (2009), 4807--4842.




\bibitem[B]{B} L. B\'egueri, {\it Dualit\'e sur un corps local
\`a corps r\'esiduel alg\'ebriquement clos},
M\'emoires de la SMF {\bf 4} (1980), 1-121.


\bibitem[Bor]{Bor} A. Borel, {\it  Linear Algebraic Groups
(Second enlarged edition)},
Graduate text in Mathematics {\bf 126} (1991), Springer.

\bibitem[BM]{BM} A.  Borel, G.D. Mostow, {\it
On semi-simple automorphisms of Lie algebras}, Ann. Math. {\bf 61}
(1955), 389-405.



\bibitem[BT1]{BT}  A. Borel, J. Tits,   {\it Groupes r\'eductifs},
    Inst. Hautes \'Etudes Sci. Publ. Math. {\bf  27} (1965),
 55--150.



\bibitem[BT2]{BT73}  A. Borel, J. Tits,   {\it Homomorphismes
``abstraits'' de groupes alg\'ebriques simples},
    Annals of Mathematics {\bf 97 } (1973), 499--571.





\bibitem[BLR]{BLR} S. Bosch, W.  L\"utkebohmert, M. Raynaud,
{\it N\'eron models},
Ergebnisse der Mathematik und ihrer Grenzgebiete {\bf  21} (1990),
Springer-Verlag.



\bibitem[Bbk1]{Bbk} N. Bourbaki, {\it  Commutative algebra}, Ch 1--7. Springer.

\bibitem[Bbk2]{Bbk1} N. Bourbaki, {\it Groupes et alg\`ebres de Lie}, Ch. 4,5 et 6, Masson.






\bibitem[BT3]{BT1} F. Bruhat, J. Tits,
{\it Groupes r\'eductifs sur un corps local.
 I. Donn\'ees radicielles valu\'ees},
 Inst. Hautes Etudes Sci. Publ. Math.  {\bf  41}  (1972), 5--251.




\bibitem[BT4]{BT2} F. Bruhat, J. Tits,
{\it Groupes r\'eductifs sur un corps local.
 II. Sch\'emas en groupes. Existence d'une donn\'ee radicielle valu\'ee},
 Inst. Hautes Etudes Sci. Publ. Math.  {\bf  60}  (1984), 197--376.


\bibitem[BT5]{BT3} F. Bruhat, J. Tits,
{\it Groupes r\'eductifs sur un corps local.
 III. Compl\'ements et applications \`a la cohomologie galoisienne},
J. Fac. Sci. Univ. Tokyo  {\bf 34} (1987), 671--698.


\bibitem[CGP]{CGP} V. Chernousov, P. Gille, A. Pianzola,
{\it Torsors over the punctured affine line}, American Journal of Mathematics 
{\bf 134} (2012), no 6, 1541--1583.

\bibitem[CTO]{CTO} J.--L. Colliot--Th\'el\`ene, M. Ojanguren, {\it Espaces
principaux homog\`enes localement triviaux},
I.H.\'E.S. Publ. Math. {\bf 75} (1992), 97--122.












\bibitem[De]{De} C. Demarche, {\it M\'ethodes cohomologiques pour l'\'etude
des points rationnels sur les espaces homog\`enes}, th\`ese (Orsay, 2009),
author's URL.



\bibitem[DG]{DG} M.  Demazure, P. Gabriel,
{\em Groupes alg\'ebriques}, North-Holland (1970).




\bibitem[G]{G} P. Gille, {\rm Le probl\`eme de Kneser-Tits}, 
expos\'e Bourbaki  983, Ast\'erisque {\bf 326}  (2009), 39--81.





\bibitem[GP1]{GP1}  P. Gille,  and  A. Pianzola,  {\it Galois cohomology
and forms of algebras over Laurent polynomial rings}, Math. Annalen
{\bf 338} (2007) 497--543.

\bibitem[GP2]{GP2} P. Gille and A. Pianzola,
{\it Isotriviality and \'etale cohomology of Laurent polynomial
rings}, Jour. Pure Applied Algebra, {\bf 212} 780--800 (2008).


\bibitem[GP3]{GP3} P. Gille and A. Pianzola,
{\it Torsors, Reductive group Schemes and Extended Affine Lie Algebras},  
Memoirs of AMS {\bf 1063} (2013).



\bibitem[Gir]{Gir} J.  Giraud,
{\em Cohomologie non-ab\'elienne}, Springer (1970).

\bibitem[Gb]{Gb} M. Greenberg,  {\it Schemata over local rings},
 Annals of Math.
{\bf 73} (1961), 624--648.


\bibitem[Gr]{FGA} A. Grothendieck,
{\it Techniques de descente et th\'eor\`emes d'existence en
g\'eom\'etrie
 alg\'ebrique. IV. Les sch\'emas de Hilbert},
S\'eminaire Bourbaki, Vol. 6,  Exp. No. 221, 249--276,
Soc. Math. France, Paris, 1995.

\bibitem[EGA IV]{EGAIV} A. Grothendieck (avec la collaboration de J.
Dieudonn\'e), {\it El\'ements de G\'eom\'etrie Alg\'ebrique IV},
Publications math\'ematiques de l'I.H.\'E.S. no 20, 24, 28 and 32
(1964 -- 1967).




\bibitem[H]{H} G. Harder, {\it  Halbeinfache Gruppenschemata
\"uber vollst\"andigen Kurven},
  Invent. Math.  {\bf 6} (1968), 107--149.





\bibitem [Hu]{[Hu]} J. Humphreys, \emph{Linear algebraic groups}, Springer-Verlag,
1975.



\bibitem[I]{I} L. Illusie, {\it Complexe de de Rham-Witt et
cohomologie cristalline}, Annales Scientifiques de l'Ecole normale
sup\'erieure  {\bf 12} (1979), 501--561.

\bibitem[Kac]{Kac} V.~Kac,
\emph{Infinite dimensional Lie algebras}, third edition,
Cambridge University Press, Cambridge, 1990.






\bibitem[Kmr]{Kmr} S. Kumar, {\it Kac-Moody groups, their flag varieties and representation theory},  Springer Verlag (2002)


\bibitem[Lam]{Lam} T.-Y. Lam,
{\em Serre's problem on projective modules}, second edition (2007),  Springer.

\bibitem[LS]{LS} S. Lang, J.-P. Serre, {\it Sur les rev\^etements non ramifi\'es des
vari\'et\'es alg\'ebriques}, American Journal of Mathematics {\bf 79} (1957), 319--330.


\bibitem[M1]{M1} J. S.  Milne,
{\em \'Etale Cohomology}, Princeton University Press.


\bibitem[M2]{M2} J. S.  Milne,
{\em Arithmetic duality theorem}, second edition (2004).

\bibitem[MP]{MP} R.V.~Moody, A.~Pianzola,
\emph{Lie algebras with
triangular decomposition},
John Wiley, New York, 1995.


\bibitem[Mt]{Mt} G. D. Mostow, {\it
Fully reducible subgroups of algebraic groups}, Amer. J. Math. {\bf
78} (1956), 200--221.


\bibitem[N1]{N1} E.~Neher, \emph{Lie tori},
C.R. Math. Acad. Sci. Soc. R. Can.,
{\bf 26} (2004), pp.~84--89.

\bibitem[N2]{N2} E.~Neher,  \emph{Extended affine Lie algebras},
C.R. Math. Acad. Sci. Soc. R. Can.,
{\bf 26} (2004), pp.~90--96.





\bibitem[P1]{P1} A. Pianzola, {\em Locally trivial principal homogeneous
spaces and conjugacy theorems for Lie algebras}, J. Algebra  {\bf 275}
(2004),  no. 2, 600--614.

\bibitem[P2]{P2} A. Pianzola, {\it Vanishing of $H^1$ for
Dedekind rings and applications to loop
algebras}, C. R. Acad. Sci. Paris, Ser. I {\bf 340} (2005), 633-638.

\bibitem[PK]{PK} D.H.~Peterson, V.~Kac,
\emph{Infinite flag varieties and conjugacy theorems},
Proc. Natl. Acad. Sci. USA \textbf{ 80} (1983), 1778--1782.






\bibitem[Pr]{Pr} G. Prasad, {\it  Galois-fixed points
in the Bruhat-Tits building of a reductive group},
Bull. Soc. Math. France  {\bf 129}, (2001),   169--174.





\bibitem[Ra]{Ra}  M. Raynaud,
{\it Anneaux locaux hens\'eliens}, Lecture Notes in Math. 169,
Springer (1971).


\bibitem[Ri]{Ri} R. W. Richardson,  {\it On orbits of algebraic
groups and Lie groups},
  Bull. Austral. Math. Soc.  {\bf 25}  (1982),  1--28.


\bibitem[Ro]{Ro} G. Rousseau, {\it Immeubles des groupes
r\'eductifs sur les corps locaux},
Th\`ese, Universit\'e de Paris-Sud (1977).




\bibitem[SGA1]{SGA1} {\it S\'eminaire de G\'eom\'etrie alg\'ebrique de
l'I.H.E.S.,  Rev\^etements \'etales et groupe fondamental, dirig\'e
par  A. Grothendieck},  Lecture Notes in Math. 224. Springer (1971).

\bibitem[SGA3]{SGA3} {\it S\'eminaire de G\'eom\'etrie alg\'ebrique de
l'I.H.E.S., 1963-1964, sch\'emas en groupes, dirig\'e par M.
Demazure et A. Grothendieck},  Lecture Notes in Math. 151-153.
Springer (1970).





\bibitem[Se1]{Se1} J.-P.~Serre, {\it Galois Cohomology},
Springer, 1997.

 \bibitem[Se2]{Se2}
 J.-P. Serre, \emph{Local fields}, Springer-Verlag, New York, 1979.






\bibitem[St]{[St75]} R. Steinberg, {\it Torsion in reductive groups},
Advances in Mathematics {\bf 15} (1975), 63--92.

\bibitem[Sz]{Sz} T. Szamuely, \emph{Galois Groups and Fundamental Groups}, 
Cambridge Studies in Advanced Mathematics, vol. 117, Cambridge University Press, 2009. 



\bibitem[Ti]{T2} J. Tits, {\it Reductive groups over local fields},
Proceedings of the Corvallis conference on
$L$-functions etc., Proc. Symp. Pure Math. {\bf 33} (1979), part 1, 29--69.







\bibitem[Y1]{Y2}
Y.~Yoshii, \emph{Lie tori---A simple characterization of extended
affine Lie algebras}, Publ. Res. Inst. Math. Sci. {\bf 42} (2006),
739--762.


\bibitem[Y2]{Y3}
Y.~Yoshii, \emph{Root systems extended by an abelian group and
their Lie algebras}, J. Lie Theory {\bf 14}
(2004), pp.~371--374.






\bibitem[Yu]{Yu} J.-K.  Yu, {\it Smooth models associated to
concave functions in Bruhat-Tits theory}, preprint (2002).



\end{thebibliography}
\end{document}